\documentclass[twoside,leqno,11pt, A4]{amsart}
\usepackage{amsfonts}
\usepackage{amsmath}
\usepackage{amscd}
\usepackage{amssymb}
\usepackage{amsthm}
\usepackage{amsrefs}
\usepackage{latexsym}
\usepackage{mathrsfs}
\usepackage{bbm}
\usepackage{enumerate}
\usepackage{color}
\begin{document}

\newtheorem{theorem}[subsection]{Theorem}
\newtheorem{proposition}[subsection]{Proposition}
\newtheorem{lemma}[subsection]{Lemma}
\newtheorem{corollary}[subsection]{Corollary}
\newtheorem{conjecture}[subsection]{Conjecture}
\newtheorem{prop}[subsection]{Proposition}
\numberwithin{equation}{section}
\newcommand{\mr}{\ensuremath{\mathbb R}}
\newcommand{\mc}{\ensuremath{\mathbb C}}
\newcommand{\dif}{\mathrm{d}}
\newcommand{\intz}{\mathbb{Z}}
\newcommand{\ratq}{\mathbb{Q}}
\newcommand{\natn}{\mathbb{N}}
\newcommand{\comc}{\mathbb{C}}
\newcommand{\rear}{\mathbb{R}}
\newcommand{\prip}{\mathbb{P}}
\newcommand{\uph}{\mathbb{H}}
\newcommand{\fief}{\mathbb{F}}
\newcommand{\majorarc}{\mathfrak{M}}
\newcommand{\minorarc}{\mathfrak{m}}
\newcommand{\sings}{\mathfrak{S}}
\newcommand{\fA}{\ensuremath{\mathfrak A}}
\newcommand{\mn}{\ensuremath{\mathbb N}}
\newcommand{\mq}{\ensuremath{\mathbb Q}}
\newcommand{\half}{\tfrac{1}{2}}
\newcommand{\f}{f\times \chi}
\newcommand{\summ}{\mathop{{\sum}^{\star}}}
\newcommand{\chiq}{\chi \bmod q}
\newcommand{\chidb}{\chi \bmod db}
\newcommand{\chid}{\chi \bmod d}
\newcommand{\sym}{\text{sym}^2}
\newcommand{\hhalf}{\tfrac{1}{2}}
\newcommand{\sumstar}{\sideset{}{^*}\sum}
\newcommand{\sumprime}{\sideset{}{'}\sum}
\newcommand{\sumprimeprime}{\sideset{}{''}\sum}
\newcommand{\shortmod}{\ensuremath{\negthickspace \negthickspace \negthickspace \pmod}}
\newcommand{\V}{V\left(\frac{nm}{q^2}\right)}
\newcommand{\sumi}{\mathop{{\sum}^{\dagger}}}
\newcommand{\mz}{\ensuremath{\mathbb Z}}
\newcommand{\leg}[2]{\left(\frac{#1}{#2}\right)}
\newcommand{\muK}{\mu_{\omega}}

\title[Low-lying zeros of quadratic and quartic Hecke $L$-functions]{One level density of low-lying zeros of quadratic and quartic Hecke $L$-functions}

\date{\today}
\author{Peng Gao and Liangyi Zhao}

\begin{abstract}
In this paper, we prove some one level density results for the low-lying zeros of families of quadratic and quartic Hecke $L$-functions of the Gaussian field.  As corollaries, we deduce that, respectively, at least $94.27 \%$ and $5\%$ of the members of the quadratic family and the quartic family do not vanish at the central point.
\end{abstract}

\maketitle



\section{Introduction}

Given a natural family of $L$-functions, the density conjecture of N. Katz and P. Sarnak \cites{KS1, K&S} states that
 the
distribution of zeros near the central point of a family of
$L$-functions is the same as that of eigenvalues near $1$ of a
corresponding classical compact group.  An important example is the family of quadratic Dirichlet characters. Let $\chi$ be a primitive Dirichlet character and we denote the non-trivial zeroes of the Dirichlet $L$-function
   $L(s, \chi)$ by $\half+i \gamma_{\chi, j}$.  Without assuming the generalized Riemann hypothesis (GRH), we order them as
\begin{equation}
\label{zeroorder}
    \ldots \leq
   \Re \gamma_{\chi, -2} \leq
   \Re \gamma_{\chi, -1} < 0 \leq \Re \gamma_{\chi, 1} \leq \Re \gamma_{\chi, 2} \leq
   \ldots.
\end{equation}
    For any primitive Dirichlet character $\chi$ with conductor $q$ of size $X$, we set
\begin{align}
\label{normalroot}
    \tilde{\gamma}_{\chi, j}= \frac{\gamma_{\chi, j}}{2 \pi} \log X
\end{align}
and define for an even Schwartz class function $\phi$,
\begin{equation} \label{Sdef}
S(\chi, \phi)=\sum_{j} \phi(\tilde{\gamma}_{\chi, j}).
\end{equation}

 For positive, odd
and square-free integers $d$, the Kronecker symbol $\chi_{8d}=(\frac{8d}{\cdot})$ is primitive.  Let
$D(X)$ denote the set of such $d$ satisfying $X\le d \le 2X$. In \cite{O&S},  A. E. \"{O}zluk and C. Snyder studied the family
of quadratic Dirichlet $L$-functions.  It follows from their work that assuming GRH for this family,
we have
\begin{align}
\label{1}
 \lim_{X \rightarrow +\infty} \frac{1}{ \# D(X)}\sum_{d \in D(X)} S(\chi_{8d}, \phi)
= \int\limits_{\rear} \phi(x) W_{USp}(x) \dif x, \; \; \mbox{where} \; \; W_{USp}(x)=1-\frac{\sin(2\pi x)}{2\pi x},
\end{align}
  provided that the support of $\hat{\phi}$, the Fourier
transform of $\phi$, is contained in the interval $(-2, 2)$.  The expression on the left-hand side of \eqref{1} is known as the one-level density of the low-lying zeros for this family of $L$-functions under consideration.\newline

The kernel of the integral $W_{USp}$ in \eqref{1} is the same function which occurs
on the random matrix theory side, when studying the eigenvalues of unitary symplectic matrices. This shows that the family of
quadratic Dirichlet $L$-functions is a symplectic family.  In \cite{Ru}, M. O. Rubinstein extended the work of A. E. \"{O}zluk and C. Snyder to all $n$-level densities (roughly speaking, investigating $n$-tuples of zeros) without assuming GRH.  He showed that the $n$-level analogue of the limit in \eqref{1} converge to the symplectic densities for test functions $\phi(x_1,\ldots, x_n)$ whose Fourier transforms $\hat{\phi}(u_1,\ldots, u_n)$ is supported in $\sum^n_{i=1}|u_i|<1$. \newline

In \cite{Gao}, assuming the truth of GRH, the first-named author computed the $n$-level
  densities for low-lying zeros of the family for
quadratic Dirichlet $L$-functions when $\hat{\phi}(u_1,\ldots, u_n)$ is supported in $\sum^n_{i=1}|u_i|<2$. It was shown that this result agrees with random matrix theory for $n \leq 7$ by J. Levinson and S. J. Miller \cite{LM} and for all $n$ by A. Entin, E. Roditty-Gershon and Z. Rudnick \cite{ER-G&R}. In \cite{MS1}, A. Mason and N. C. Snaith developed a new formula for the $n$-level densities of eigenvalues of unitary symplectic matrices and they show in \cite{MS} that their result leads to a relatively straightforward way to match the result in \cite{Gao} with random matrix theory. \newline

 Besides the family of quadratic Dirichlet $L$-functions, the density conjecture has been confirmed
for many other families of $L$-functions, such as different types of
Dirichlet $L$-functions \cites{Gao, O&S, Ru, Miller1, HuRu}, $L$-functions with
characters of the ideal class group of the imaginary quadratic
field $\ratq(\sqrt{-D})$ \cite{FI}, automorphic $L$-functions \cites{ILS, DuMi2, HuMi, RiRo, Royer}, elliptic curves $L$-functions \cites{SBLZ1, HB1, Brumer, SJM, Young}, symmetric powers of $GL(2)$ $L$-functions \cite{Gu2, DM} and a family of $GL(4)$ and $GL(6)$
$L$-functions \cite{DM}. \newline

   Among the many results concerning the $n$-level
  densities of low-lying zeroes for various families of $L$-functions,  A. M. G\"ulo\u{g}lu \cite{Gu} studied the one-level density of the low-lying zeros of
a family of Hecke $L$-functions of $\ratq (\omega)$ ($\omega
= \exp (2 \pi i/3)$) associated with cubic symbols
$\chi_c=(\frac {\cdot}{c})_3$ with $c$ square-free and congruent
to 1 modulo 9, regarded as primitive ray class characters of the
ray class group $h_{(c)}$.  We recall here that for any $c$, the
ray class group $h_{(c)}$ is defined to be $I_{(c)}/P_{(c)}$,
where $I_{(c)} = \{ \mathcal{A} \in I, (\mathcal{A}, (c)) = 1 \}$
and $P_{(c)} = \{(a) \in P, a \equiv 1 \pmod{c} \}$ with $I$ and
$P$ denoting the group of fractional ideals in $\ratq (\omega)$
and the subgroup of principal ideals, respectively.  Assuming GRH, G\"ulo\u{g}lu \cite{Gu} obtained the one-level density for this family when the Fourier transform of the test function is supported in $(-31/30, 31/30)$. In \cite{G&Zhao2}, an unconditional result was obtained for the same family with a more restricted range for the support of the Fourier transform of the test function. \newline

A general approach towards establishing the $n$-level densities involves converting the sum over
zeros of the $L$-functions under consideration into sums over
primes using the relevant versions of the explicit formula (see Section \ref{section: Explicit Formula}).  This
leads to estimation on certain character sums. It is this estimation that mainly affects the breadth of the support of the Fourier transform of the test function in the resulting expression. A common key ingredient used both in \cite{Gao} and \cite{Gu} is the Poisson summation to, essentially, convert the corresponding character sums to other character sums over dual lattices. In this process, the length of the character sum is shortened which allows one to get a better estimation resulting in the enlargement on the support of the Fourier transform of the test function in the expression for $n$-level densities. In \cite{Gao}, the Poisson summation over $\mz$ is used,  following a method of K. Soundararajan \cite{sound1}. In \cite{Gu}, a two dimensional Poisson summation is used, which is similar to the Poisson summation over $\mz(\omega)$ developed by D. R. Heath-Brown in \cite[Lemma 10]{H}.
\newline

   Motivated by the results in \cite{Gu} and \cite{Gao}, it is our goal in the paper to further explore the application of Poisson summation in the study of one level density results for the
low-lying zeros.  We focus our attention on the family of quadratic and quartic Hecke $L$-functions in the Gaussian field $K=\mq(i)$. \newline

 Let $\chi$ be a primitive Hecke character, the Hecke $L$-function associated with
$\chi$ is defined for $\Re(s) > 1$ by
\begin{equation*}
  L(s, \chi) = \sum_{0 \neq \mathcal{A} \subset
  \mathcal{O}_K}\chi(\mathcal{A})(N(\mathcal{A}))^{-s},
\end{equation*}
  where $\mathcal{A}$ runs over all non-zero integral ideals in $K$ and $N(\mathcal{A})$ is the
norm of $\mathcal{A}$. As shown by E. Hecke, $L(s, \chi)$ admits
analytic continuation to an entire function and satisfies a
functional equation.  We refer the reader to \cites{Gu, Luo, G&Zhao1, G&Zhao3} for a more detailed discussion of these Hecke characters and $L$-functions.  We denote non-trivial zeroes of $L(s,
\chi)$ by $\half+i \gamma_{\chi, j}$ and order them in a
fashion similar to \eqref{zeroorder}.  \newline

Set
\[ C(X) = \{ c \in \intz[i] : (c,1+i) =1, \; c \; \mbox{square-free}, \; X \leq N(c) \leq 2X \} . \]
We shall define in Section \ref{sect: Kronecker} the primitive quadratic Kronecker symbol $\chi_{i(1+i)^5c}$ and the primitive quartic Kronecker symbol $\chi_{(1+i)^7c}$.  For $\chi=\chi_{i(1+i)^5c}$ or $\chi_{(1+i)^7c}$, we set $\tilde{\gamma}_{\chi, j}$ as in \eqref{normalroot} and $S(\chi, \phi)$ as in \eqref{Sdef} for an even Schwartz class function $\phi$.  Further, let $\Phi_X(t)$ be a non-negative smooth function supported on $(1,2)$,
    satisfying $\Phi_X(t)=1$ for $t \in (1+1/U, 2-1/U)$ with $U=\log \log X$ and such that
    $\Phi^{(j)}_X(t) \ll_j U^j$ for all integers $j \geq 0$. Our results are as follows.
\begin{theorem}
\label{quadraticmainthm}
Suppose that GRH is true.  Let $\phi(x)$ be an even Schwartz function whose
Fourier transform $\hat{\phi}(u)$ has compact support in $(-2,2)$, then
\begin{align}
\label{quaddensity}
 \lim_{X \rightarrow +\infty}\frac{1}{\# C(X)}\sumstar_{(c, 1+i)=1}  S(\chi_{i(1+i)^5c}, \phi)\Phi_X \left( \frac {N(c)}{X} \right)
 = \int\limits_{\mathbb{R}} \phi(x) \left(1-\frac{\sin(2\pi x)}{2\pi x} \right)\dif x.
\end{align}
   Here the ``$*$'' on the sum over $c$ means that the sum is restricted to square-free elements $c$ of $\mathbb{Z}[i]$ .
\end{theorem}

S. Chowla \cite{chow} conjectured that $L(1/2, \chi) \neq 0$ for any primitive Dirichlet character $\chi$.  One expects that the same statement should hold for Hecke characters as well.  Using Theorem~\ref{quadraticmainthm}, we can deduce the following non-vanishing result.

  \begin{corollary} \label{quadnonvan}
 Suppose that the GRH is true and that $1/2$ is a zero of $L \left( s, \chi_{i(1+i)^5c} \right)$ of order $m_c \geq 0$.  As $X \to \infty$,
\[  \sumstar_{(c, 1+i)=1} m_c \Phi_X \left( \frac {N(c)}{X} \right) \leq  \left( \frac{\cot \frac{1}{4} -3 }{8} + o(1) \right) \# C(X). \]
   Moreover, as $X \to \infty$
\[ \# \{ c \in C(X) : L\left( 1/2, \chi_{i(1+i)^5c} \right) \neq 0 \} \geq \left( \frac{19-\cot \frac{1}{4}}{16} + o(1) \right) \# C(X) . \]
    \end{corollary}

\begin{proof}
The proof goes along the same line as that of \cite[Corollary 2.1]{B&F}.  We thus omit the details here.
\end{proof}

For a family of quartic characters, we have the following

\begin{theorem}
\label{quarticmainthm}
Suppose that GRH is true.  Let $\phi(x)$ be an even Schwartz function whose
Fourier transform $\hat{\phi}(u)$ has compact support in $(-20/19, 20/19)$, then
\begin{align} \label{quarticmainthmeq}
 \lim_{X \rightarrow +\infty}\frac{1}{\# C(X)}\sumstar_{(c, 1+i)=1}  S(\chi_{(1+i)^7c}, \phi)\Phi_X \left( \frac {N(c)}{X} \right)
 = \int\limits_{\mathbb{R}} \phi(x) \dif x.
\end{align}
   Here the ``$*$'' on the sum over $c$ means that the sum is restricted to square-free elements $c$ of $\mathbb{Z}[i]$ .
\end{theorem}

Similar to Corollary~\ref{quadnonvan}, we have a non-vanishing result for the family of quartic Hecke $L$-functions under our consideration.

 \begin{corollary}
 Suppose that the GRH is true and that $1/2$ is a zero of $L \left( s, \chi_{(1+i)^7c} \right)$ of order $n_c \geq 0$.  As $X \to \infty$,
 \begin{equation} \label{coreq}
   \sumstar_{(c, 1+i)=1} n_c \Phi_X \left( \frac {N(c)}{X} \right) \leq  \left( \frac{19}{20} + o(1) \right) \# C(X).
   \end{equation}
   Moreover, as $X \to \infty$
\begin{equation} \label{coreq2}
  \# \{ c \in C(X) : L\left( 1/2, \chi_{(1+i)^7c} \right) \neq 0 \} \geq \left( \frac{1}{20} + o(1) \right) \# C(X) .
  \end{equation}
    \end{corollary}

\begin{proof}
Consider
\[ \phi_0(x) = \left( \frac{\sin \pi x}{\pi x} \right)^2 . \]
It is well-known that
\[ \int\limits_{\rear} \phi_0(x) \dif x = 1 \quad \mbox{and} \quad \hat{\phi}_0 (x) = \max \{ |x| , 0 \} . \]
Take $0 \leq \theta < 20/19$, then $\phi(x) = \phi_0(\theta x)$ satisfies the requirements of Theorem~\ref{quarticmainthm}.  The truth of GRH implies that all the non-trivial zeros of $L(s,\chi_{(1+i)^7c})$ are of the form $1/2+i \gamma$ with $\gamma \in \rear$.  So
\[ n_c \leq S(\chi_{(1+i)^7c}, \phi) \]
for every $c$.   Now using \eqref{quarticmainthmeq}, we get that as $X \to \infty$
 \[  \frac{1}{\# C(X)} \sumstar_{(c, 1+i)=1} n_c \Phi_X \left( \frac {N(c)}{X} \right) \leq \frac{1}{\# C(X)}\sumstar_{(c, 1+i)=1}  S(\chi_{(1+i)^7c}, \phi)\Phi_X \left( \frac {N(c)}{X} \right) = \int\limits_{\rear} \phi(x) \dif x + o(1) . \]
By taking $\theta$ arbitrarily close to $20/19$, the last integral above is $19/20 + o(1)$ and \eqref{coreq} follows from this. \newline

To prove \eqref{coreq2}, we start with
\[ \# C(X) =   \# \{ c \in C(X) : L\left( 1/2, \chi_{(1+i)^7c} \right) \neq 0 \} +   \# \{ c \in C(X) : L\left( 1/2, \chi_{(1+i)^7c} \right) = 0 \} . \]
As $X \to \infty$,
\[ \# \{ c \in C(X) : L\left( 1/2, \chi_{(1+i)^7c} \right) = 0 \} \leq \sumstar_{(c, 1+i)=1} \left( n_c \Phi_X \left( \frac {N(c)}{X} \right) +o(1) \right) \leq  \left( \frac{19}{20} + o(1) \right) \# C(X), \]
using \eqref{coreq}.  Now \eqref{coreq2} follows easily from the above.
\end{proof}

   Note that it follows from \eqref{1} that Theorem \ref{quadraticmainthm} shows that the family of quadratic Hecke $L$-functions is a symplectic family. We also note that, as
\[ \int\limits_{\rear} f(x) \dif x=\int\limits_{\rear} f(x)W_U(x) \dif x, \; \mbox{with} \; W_U(x)=1, \]
Theorem \ref{quarticmainthm} gives that the family of quartic Hecke $L$-functions is a unitary family. This is already observed in \cite{G&Zhao2}, where an unconditional result is obtained for the family of quartic Hecke $L$-functions with smaller support of the Fourier transform of the test function.

\subsection{Notations} The following notations and conventions are used throughout the paper.\\
\noindent $\Phi(t)$ for $\Phi_X(t)$. \newline
\noindent $e(z) = \exp (2 \pi i z) = e^{2 \pi i z}$. \newline
$f =O(g)$ or $f \ll g$ means $|f| \leq cg$ for some unspecified
positive constant $c$. \newline
$f =o(g)$ means $\displaystyle \lim_{x \rightarrow \infty}f(x)/g(x)=0$. \newline
$\mu_{[i]}$ denotes the M\"obius function on $\mz[i]$. \newline
The letter $\varpi$ reserved for primes in $\mz[i]$. \newline
$\zeta_{\mq(i)}(s)$ denotes the Dedekind zeta function of $\mq(i)$. \newline
$\chi_{[-1,1]}$ denotes the characteristic function of $[-1,1]$.

\section{Preliminaries}
\label{sec 2}

\subsection{Quadratic, quartic characters and Kronecker symbols}
\label{sect: Kronecker}
     The symbol $\leg {\cdot}{n}_4$ is the quartic
residue symbol in the ring $\mz[i]$.  For a prime $\varpi \in \mz[i]$
with $N(\varpi) \neq 2$, the quartic character is defined for $a \in
\mz[i]$, $(a, \varpi)=1$ by $\leg{a}{\varpi}_4 \equiv
a^{(N(\varpi)-1)/4} \pmod{\varpi}$, with $\leg{a}{\varpi}_4 \in \{
\pm 1, \pm i \}$. When $\varpi | a$, we define
$\leg{a}{\varpi}_4 =0$.  Then the quartic character can be extended
to any composite $n$ with $(N(n), 2)=1$ multiplicatively. We extend the definition of $\leg{\cdot }{n}_4$ to $n=1$ by setting $\leg{\cdot}{1}_4=1$. We further define $(\frac{\cdot}{n})=\leg {\cdot}{n}^2_4$  to be the quadratic
residue symbol for these $n$.  \newline

 Note that in $\intz[i]$, every ideal co-prime to $2$ has a unique
generator congruent to 1 modulo $(1+i)^3$.  Such a generator is
called primary. Recall that \cite[Theorem 6.9]{Lemmermeyer} the quartic reciprocity law states
that for two primary integers  $m, n \in \mz[i]$,
\begin{equation*}
 \leg{m}{n}_4 = \leg{n}{m}_4(-1)^{((N(n)-1)/4)((N(m)-1)/4)}.
\end{equation*}

    As a consequence, the following quadratic reciprocity law holds for two primary integers  $m, n \in \mz[i]$:
\begin{align}
\label{quadrecip}
 \leg{m}{n} = \leg{n}{m}.
\end{align}

   Observe that a non-unit
$n=a+bi$ in $\mz[i]$ is congruent to $1
\bmod{(1+i)^3}$ if and only if $a \equiv 1 \pmod{4}, b \equiv
0 \pmod{4}$ or $a \equiv 3 \pmod{4}, b \equiv 2 \pmod{4}$ by Lemma
6 on page 121 of \cite{I&R}. \newline

    From the supplement theorem to the quartic reciprocity law (see for example, Lemma 8.2.1 and Theorem 8.2.4 in \cite{BEW}),
we have for $n=a+bi$ being primary,
\begin{align}
\label{2.05}
  \leg {i}{n}_4=i^{(1-a)/2} \qquad \mbox{and} \qquad  \hspace{0.1in} \leg {1+i}{n}_4=i^{(a-b-1-b^2)/4}.
\end{align}

  We now define a character of order $2$ modulo $(1+i)^5$. In fact, for any element $c \in \mz[i], (c, 1+i)=1$, we can define a Dirichlet character $\chi_{i(1+i)^5c} \pmod {(1+i)^5c}$ by noting that the ring
$(\mz[i]/(1+i)^5c\mz[i])^*$ is isomorphic to the direct product of the group of units $U=\langle i \rangle$ and the group $N_{(1+i)^5c}$ formed by elements in $(\mz[i]/(1+i)^5c\mz[i])^*$ congruent to  $1 \pmod {(1+i)^3}$ (i.e. primary). Under this isomorphism, any element $n \in (\mz[i]/(1+i)^5c\mz[i])^*$ can be written uniquely as $n=u_n \cdot n_0$ with $u_n \in U$, $n_0 \in N_{(1+i)^5c}$.  We can now define $\chi_{i(1+i)^5c} \pmod {(1+i)^5c}$ such that for any $n \in (\mz[i]/(1+i)^5c\mz[i])^*$,
\begin{align*}
   \chi_{i(1+i)^5c}(n)=\leg {i(1+i)^5c}{n_0}.
\end{align*}

    One deduces from \eqref{2.05} and the quadratic reciprocity that  $\chi_{i(1+i)^5c}(n)=1$ when $n_0 \equiv 1 \pmod {(1+i)^5c}$. It follows from this that  $\chi_{i(1+i)^5c}(n)$ is well-defined (i.e. $\chi_{i(1+i)^5c}(n)=\chi_{i(1+i)^5c}(n')$ when $n \equiv n' \pmod {(1+i)^5c}$). As $\chi_{i(1+i)^5c}(n)$ is clearly multiplicative, of order $2$ and trivial on units, it can be regarded as a primitive Hecke character $\pmod {(1+i)^5c}$ of trivial infinite type.  Let $\chi_{i(1+i)^5c}$ stand for this Hecke character as well and we call it the Kronecker symbol. Furthermore, when $c$ is square-free, $\chi_{i(1+i)^5c}$ is non-principal and primitive. To see this, we write $c=u_c \cdot \varpi_1 \cdots \varpi_k$ with $u_c \in U$ and $\varpi_j$  being primary primes. Suppose $\chi_{i(1+i)^5c}$ is induced by some $\chi$ modulo $c'$ with $\varpi_j \nmid c'$, then by the Chinese Remainder Theorem, there exists an $n$ such that $n \equiv 1 \pmod {(1+i)^5c/\varpi_j}$ and $\leg {n}{\varpi_j} \neq 1$. It follows that $\chi(n)=1$ but $\chi_{i(1+i)^5c}(n) \neq 1$, a contradiction. Thus, $\chi_{i(1+i)^5c}$ can only be possibly induced by some $\chi$ modulo $(1+i)^4c$. By the Chinese Remainder Theorem, there exists an $n$ such that $n \equiv 1 \pmod {c}$ and $n \equiv 5 \pmod {(1+i)^5}$. As this $n \equiv 1 \pmod {(1+i)^4}$, it follows that $n \equiv 1 \pmod {(1+i)^4c}$, hence $\chi(n)=1$ but $\chi_{i(1+i)^5c}(n)=-1 \neq 1$ (note that $\leg {i}{n}=1$ when $n \equiv 5 \pmod {(1+i)^5}$ so we may assume that $c$ is primary) and this implies that $\chi_{i(1+i)^5c}$ is primitive. This also shows that $\chi_{i(1+i)^5c}$ is non-principal. \newline

    Similarly, for any element $c \in \mz[i], (c, 1+i)=1$, we define the Kronecker symbol $\chi_{(1+i)^7c}$ as a character of order $4$ modulo $(1+i)^7c$ such that for any $n \in (\mz[i]/(1+i)^7c\mz[i])^*$, with $n=u_n \cdot n_0$, $u_n \in U$ and $n_0$ being primary,
\begin{align*}
   \chi_{(1+i)^7c}(n)=\leg {(1+i)^7c}{n_0}_4.
\end{align*}

\subsection{The Gauss sums}
\label{section:Gauss}

For any $n \in \mz[i]$, $n \equiv 1 \pmod {(1+i)^3}$, the quadratic and quartic
 Gauss sums $g_2(n)$, $g_4(n)$ are defined by
\[    g_2(n) =\sum_{x \bmod{n}} \leg{x}{n} \widetilde{e}\leg{x}{n} \qquad \mbox{and} \qquad
   g_4(n)  =\sum_{x \bmod{n}} \leg{x}{n}_4 \widetilde{e}\leg{x}{n}, \]
  where $ \widetilde{e}(z) =\exp \left( 2\pi i  (\frac {z}{2i} - \frac {\overline{z}}{2i}) \right)$.  Note that $g_2(1)=g_4(1)=1$ by definition. \newline

  More generally, for any $n \in \mz[i]$, $n \equiv 1 \pmod {(1+i)^3}$, we set
\begin{align*}
 g_2(r,n) = \sum_{x \bmod{n}} \leg{x}{n} \widetilde{e}\leg{rx}{n} \qquad \mbox{and} \qquad g_4(r,n) = \sum_{x \bmod{n}} \leg{x}{n}_4 \widetilde{e}\leg{rx}{n}.
\end{align*}

     The following properties of $g_4(r,n)$ can be found in \cite{Diac}:
\begin{lemma} \label{quarticGausssum}
   We have
\begin{align}
\label{eq:gmult}
 g_4(rs,n) & = \overline{\leg{s}{n}}_4 g_4(r,n), \quad (s,n)=1, \\
\label{2.03}
   g_4(r,n_1 n_2) &=\leg{n_2}{n_1}_4\leg{n_1}{n_2}_4g_4(r, n_1) g_4(r, n_2), \quad (n_1, n_2) = 1, \\
\label{2.04}
g_4(\varpi^k, \varpi^l)& =\begin{cases}
    N(\varpi)^kg_4(\varpi) \qquad & \text{if} \qquad l= k+1, k \equiv 0 \pmod {4},\\
    N(\varpi)^kg_2(\varpi) \qquad & \text{if} \qquad l= k+1, k \equiv 1 \pmod {4},\\
    N(\varpi)^k\leg{-1}{\varpi}_4\overline{g_4}(\varpi) \qquad & \text{if} \qquad l= k+1, k \equiv 2 \pmod {4},\\
    -N(\varpi)^k, \qquad & \text{if} \qquad l= k+1, k \equiv 3 \pmod {4},\\
      \varphi(\varpi^l)=\#(\mz[i]/(\varpi^l))^* \qquad & \text{if} \qquad  k \geq l, l \equiv 0 \pmod {4},\\
      0 \qquad & \text{otherwise}.
\end{cases}
\end{align}
\end{lemma}

   Similarly, the next lemma allows us to evaluate $g_2(r,n)$ for $n \equiv 1 \pmod {(1+i)^3}$ explicitly.
\begin{lemma} \label{Gausssum}
\begin{enumerate}[(i)]
\item  We have
\begin{align}
\label{2.7}
 g_2(rs,n) & = \overline{\leg{s}{n}} g_2(r,n), \qquad (s,n)=1, \\
   g_2(k,mn) & = g_2(k,m)g_2(k,n),   \qquad  m,n \text{ primary and } (m , n)=1 .
\end{align}
\item Let $\varpi$ be a primary prime in $\mz[i]$. Suppose $\varpi^{h}$ is the largest power of $\varpi$ dividing $k$. (If $k = 0$ then set $h = \infty$.) Then for $l \geq 1$,
\begin{align*}
g_2(k, \varpi^l)& =\begin{cases}
    0 \qquad & \text{if} \qquad l \leq h \qquad \text{is odd},\\
    \varphi(\varpi^l)=\#(\mz[i]/(\varpi^l))^* \qquad & \text{if} \qquad l \leq h \qquad \text{is even},\\
    -N(\varpi)^{l-1} & \text{if} \qquad l= h+1 \qquad \text{is even},\\
    \leg {ik\varpi^{-h}}{\varpi}N(\varpi)^{l-1/2} \qquad & \text{if} \qquad l= h+1 \qquad \text{is odd},\\
    0, \qquad & \text{if} \qquad l \geq h+2.
\end{cases}
\end{align*}
\end{enumerate}
\end{lemma}
\begin{proof}
(i) The proof is similar to those of \eqref{eq:gmult} and \eqref{2.03}, using the quadratic reciprocity \eqref{quadrecip}. \newline

(ii)
  The case $l \leq h$ is easily verified. If $l > h$, then
\begin{align}
\label{2.8}
    \sum_{a \bmod {\varpi^{l}}}\leg {a}{\varpi^l}\widetilde{e}\left(\frac {ka}{\varpi^l}\right) =  \sum_{b \bmod \varpi}\leg {b}{\varpi^l}\sum_{c
    \bmod {\varpi^{l-1}}}\widetilde{e}\left(\frac {k(c\varpi+b)}{\varpi^l}\right).
\end{align}

    We write the inner sum above as
\begin{align*}
     \widetilde{e}\left(\frac {kb}{\varpi^l}\right)\sum_{c
    \bmod {\varpi^{l-1}}}\widetilde{e}\left(\frac {k\varpi^{-h}c}{\varpi^{l-h-1}}\right)=\widetilde{e}\left(\frac {kb}{\varpi^l}\right)\sum_{c
    \bmod {\varpi^{l-1}}}\widetilde{e}\left(\frac {c}{\varpi^{l-h-1}}\right).
\end{align*}

    When $l \geq h+2$, we write $c = c_1 \varpi^{l-h-1} + c_2$ where $c_1$ varies over a set of
representatives in $\mz[i] \pmod{\varpi^h}$ and $c_2$ varies over a set
of representatives in $\mz[i] \pmod{\varpi^{l-h-1}}$ to see that
\begin{align*}
   \sum_{c \bmod {\varpi^{l-1}}}\widetilde{e}\left(\frac {c}{\varpi^{l-h-1}}\right)
   = N(\varpi^h)
   \sum_{c_2 \bmod {\varpi^{l-h-1}}}\widetilde{e} \left( \frac{c_2}{\varpi^{l-h-1}} \right).
\end{align*}
   Now we can find a $c_3$ such that $\widetilde{e}(c_3/\varpi^{l-h-1}) \neq 1$ (for example, take $c_3=1$ when $\varpi$ is not rational and $c_3=i$ when $\varpi$ is rational) to deduce that
\begin{align*}
  \widetilde{e} \left( \frac{c_3}{\varpi^{l-h-1}} \right) \sum_{c_2 \bmod {\varpi^{l-h-1}}}\widetilde{e} \left( \frac{c_2}{\varpi^{l-h-1}} \right)=\sum_{c_2 \bmod {\varpi^{l-h-1}}}\widetilde{e} \left( \frac{c_2+c_3}{\varpi^{l-h-1}} \right)=\sum_{c_2 \bmod {\varpi^{l-h-1}}}\widetilde{e} \left( \frac{c_2}{\varpi^{l-h-1}} \right).
\end{align*}
  This implies that
\begin{align}
\label{2.9}
   \sum_{c_2 \bmod {\varpi^{l-h-1}}}\widetilde{e} \left( \frac{c_2}{\varpi^{l-h-1}} \right)=0.
\end{align}
   This proves the last case when $l \geq h+2$. \newline

  When $l=h+1$, the right-hand side expression of \eqref{2.8} is
\begin{align*}
    N(\varpi)^{l-1}\sum_{b \bmod \varpi}\leg {b}{\varpi^l}\widetilde{e}\left(\frac {kb}{\varpi^l}\right).
\end{align*}
   If $l$ is even then the last sum above is $-1$ (using \eqref{2.9}) and if $l$ is odd the last sum above is
\begin{align*}
    \sum_{b \bmod \varpi}\leg {b}{\varpi}\widetilde{e}\left(\frac {b(k\varpi^{-h})}{\varpi}\right)=\leg {k\varpi^{-h}}{\varpi}g_2(\varpi)=\leg {ik\varpi^{-h}}{\varpi}N(\varpi)^{1/2},
\end{align*}
where the expression of $g_2(\varpi)$ follows from \cite[Proposition 2.2]{Onodera} (be aware that the definition of the Gauss sum in \cite{Onodera} is different from the one here) and this completes the proof of the lemma.
\end{proof}

\subsection{The Explicit Formula}
\label{section: Explicit Formula}

Our approach in this paper relies on the following explicit formula, which essentially converts a sum over zeros of an
$L$-function to a sum over primes.

\begin{lemma}
\label{lem2.4}
   Let $\phi(x)$ be an even Schwartz function whose Fourier transform
   $\hat{\phi}(u)$ has compact support. Let $c \in \mz[i]$ be square-free satisfying $(c, 1+i)=1, X \leq N(c) \leq 2X$ and let $\chi=\chi_{i(1+i)^5c}$ or $\chi_{(1+i)^7c}$. We have
\begin{equation*}
S(\chi, \phi) =\int\limits^{\infty}_{-\infty}  \phi(t) \dif t-\sum^{2}_{j=1} \left( S_j(\chi,X;\hat{\phi})+S_j(\overline{\chi},X;\hat{\phi}) \right)+O\left(\frac{1}{\log
X}\right),
\end{equation*}
   where
\begin{equation*}
   S_j(\chi,X;\hat{\phi})=\frac 1{\log X}\sum_{\varpi \equiv 1 \bmod {(1+i)^3}}\frac { \log
   N(\varpi)}{\sqrt{N(\varpi^j)}}\chi(\varpi^j)  \hat{\phi}\left( \frac {\log N(\varpi^j)}{\log X} \right)
\end{equation*}
with the sum over $\varpi$ running over primes in $\mz[i]$.
\end{lemma}

\begin{proof} The proof is rather standard and goes along the same line as \cite[Lemma 4.1]{Gu}. \end{proof}

    We write $S(\chi,X;\hat{\phi})$ for $S_1(\chi,X;\hat{\phi})$ in the rest of the paper and we deduce from Lemma \ref{lem2.4} the following
\begin{lemma}
\label{lem2.1}
   Let $\phi(x)$ be an even Schwartz function whose Fourier transform
   $\hat{\phi}(u)$ is compactly supported. For any square-free $c \in \mz[i], (c, 1+i)=1, X \leq N(c) \leq 2X$, we have
\begin{align*}
  S(\chi_{i(1+i)^5c}, \phi)   =\int\limits^{\infty}_{-\infty}\phi(t) \dif t-\frac 1{2}
   \int\limits^{\infty}_{-\infty}\hat{\phi}(u) \dif u-2S(\chi_{i(1+i)^5c},X; \hat{\phi})+O \left( \frac {\log \log 3X}{\log
   X} \right),
\end{align*}
    with the implicit constant depending on $\phi$.
\end{lemma}
\begin{proof}
    Note first that
\begin{equation*}
  \sum_{\varpi | i(1+i)^5c}\frac {\log N(\varpi)}{N(\varpi)} \ll \log \log 3X.
\end{equation*}

   It follows that
\begin{equation*}
\begin{split}
-S_2(\chi_{i(1+i)^5c},X;\hat{\phi})-S_2(\overline{\chi}_{i(1+i)^5c},X;\hat{\phi})
= - \frac 2{\log X}\sum_{\varpi \equiv 1 \bmod {(1+i)^3}} & \frac {\log N(\varpi)}{N(\varpi)}
   \hat{\phi}\left (\frac {2 \log N(\varpi)}{\log X} \right)  \\
   & +O \left( \frac {\log \log 3X}{\log X} \right).
   \end{split}
\end{equation*}
The prime ideal theorem \cite[Theorem 8.9]{MVa1}, together with partial summation, gives that for $x \geq 1$,
\begin{equation} \label{mer}
    \sum_{\substack{ N(\varpi) \leq x \\ \varpi \equiv 1 \bmod {(1+i)^3}}} \frac {\log N(\varpi)}{N(\varpi) }= \log x+O(\log \log 3x).
\end{equation}

    From this and partial summation, we see that
\begin{align*}
 -S_2(\chi_{i(1+i)^5c},X;\hat{\phi})-S_2(\overline{\chi}_{i(1+i)^5c},X;\hat{\phi}) & =  -\frac 2{\log X}\int\limits^{\infty}_{1}\hat{\phi}\left( \frac {2\log t}{\log X}\right) \frac {\dif t}{t}+O\left( \frac {\log \log 3X}{\log X} \right) \\
 & = -\frac 1{2}
\int\limits^{\infty}_{-\infty}\hat{\phi}(t)\dif t+O \left( \frac {\log \log 3X}{\log
    X} \right).
\end{align*}
  The assertion of the lemma follows from this and Lemma \ref{lem2.4}.
\end{proof}

    To estimate the terms $S_2(\chi,X;\hat{\phi})$ and $S_2(\overline{\chi},X;\hat{\phi})$ in Lemma \ref{lem2.4} for $\chi=\chi_{(1+i)^7c}$, we need the following
\begin{lemma}
\label{lem2.7}
Suppose that GRH is true. For any non-principal Hecke character $\chi$ of trivial infinite type with modulus $n$, we have for $x \geq 1$,
\begin{align} \label{lem2.7eq}
  S(x, \chi)=\sum_{\substack {N(\varpi) \leq x \\ \varpi \equiv 1 \bmod {(1+i)^3}}} \chi (\varpi) \log N(\varpi)
\ll \min \left\{ x, \sqrt{x} \log^{3} x \log N(n) \right\}.
\end{align}
\end{lemma}
\begin{proof}
The proof of this Lemma is rather standard and we give a sketch here.  The details can be found in the proof of \cite[Lemma 2.19]{Gu}.  First, it follows from the prime ideal theorem \cite[Theorem 8.9]{MVa1} and partial summation that
\[ S(x, \chi) \ll x. \]
Now using Perron's formula with $a=1+1/\log x$, we have
\begin{equation} \label{afterperron}
 - \frac{1}{2\pi i} \int\limits_{a-i\sqrt{x}}^{a+i\sqrt{x}} \frac{L'(s,\chi)}{L(s,\chi)} \frac{x^s}{s} \dif s - S(x,\chi) \ll \sqrt{x} \log x + \sqrt{x} \sum_{\mathfrak{a}} \frac{\Lambda(\mathfrak{a})}{N(\mathfrak{a})^a | \log (x/N(\mathfrak{a}))|} ,
 \end{equation}
where $\Lambda(\mathfrak{a})$ is the analogue of the von Mangoldt function in $\mz[i]$.  Breaking up the sum over $\mathfrak{a}$ into the ranges
\[ N(\mathfrak{a}) \leq  x/2, \; x/2 < N(\mathfrak{a}) \leq x, \; x < N(\mathfrak{a}) \leq 3x/2, \; 3x/2 < N(\mathfrak{a}) \]
and using different lower bounds for $|\log (x/N(\mathfrak{a}))|$ for each of these ranges, we get that
\[  \sum_{\mathfrak{a}} \frac{\Lambda(\mathfrak{a})}{N(\mathfrak{a})^a | \log (x/N(\mathfrak{a}))|} \ll \log^2 x . \]
Now assuming the truth of GRH, we can move the line of integration in \eqref{afterperron} to $1/2+1/\log x$ without picking up any residue.  Applying standard bounds for $L'(s,\chi)/L(s,\chi)$ (see \cite[Lemma 2.18]{Gu}) implies that the integral in \eqref{afterperron} is
\[ \ll \sqrt{x} \log^3 x \log N(n), \]
which gives the second term inside the minimum in \eqref{lem2.7eq} and completes the proof.
\end{proof}

    Applying Lemma \ref{lem2.7}, we see that term corresponding to the second sum in the expression of $S(\chi, \phi)$ in Lemma \ref{lem2.4} for $\chi=\chi_{(1+i)^7c}$ (note that in this case $\chi$ is non-principal) contributes
\begin{align*}
 \frac 1{\log X}\int\limits^{\infty}_{1} \frac 1{t} & \phi \left( \frac {2\log t}{\log X} \right) \dif S(t, \chi) \\
  &  \ll \frac 1{\log X}\int\limits^{\infty}_{1} S(t,\chi) \frac 1{t^2} \phi \left( \frac {2\log t}{\log X} \right) \dif t+ \frac 1{\log^2 X}\int\limits^{\infty}_{1} S(t, \chi) \frac 1{t^2}\phi' \left( \frac {2\log t}{\log X} \right) \dif t \\
    & \ll \frac 1{\log X}\int\limits^{\log^3 N(c) }_{1} \frac 1t \dif t+ \frac {\log N(c)}{\log X}\int\limits^{\infty}_{\log^3 N(c)} t^{-3/2+\varepsilon} \dif t \ll \frac {\log \log N(c)}{\log X}.
\end{align*}

We then arrive at the following
\begin{lemma}
\label{lem2.8}
   Let $\phi(x)$ be an even Schwartz function whose Fourier transform
   $\hat{\phi}(u)$ has compact support. Let $\chi=\chi_{(1+i)^7c}$ for any square-free $c \in \mz[i], (c, 1+i)=1, X \leq N(c) \leq 2X$.  We have
\begin{equation*}
S(\chi, \phi) =\int\limits^{\infty}_{-\infty}  \phi(t) \dif t-S(\chi_{i(1+i)^7c},X;\hat{\phi})-S(\overline{\chi}_{(1+i)^7c},X;\hat{\phi}) +O\left(\frac{\log \log 3X}{\log
X}\right),
\end{equation*}
   with the implicit constant depending on $\phi$.
\end{lemma}

\subsection{Poisson Summation}
The proofs of Theorems \ref{quadraticmainthm} and \ref{quarticmainthm} require the following Poisson summation formula.
\begin{lemma}
\label{Poissonsum} Let $n \in \mz[i],  n \equiv 1 \pmod {(1+i)^3}$ and $\chi$ a quadratic or quartic character $\pmod {n}$. For any Schwartz class function $W$,  we have for all $a>0$,
\begin{align*}
   \sum_{m \in \mz[i]}\chi(m)W\left(\frac {aN(m)}{X}\right)=\frac {X}{aN(n)}\sum_{k \in
   \mz[i]}g(k,n)\widetilde{W}\left(\sqrt{\frac {N(k)X}{aN(n)}}\right),
\end{align*}
   where
\begin{align*}
   \widetilde{W}(t) &=\int\limits^{\infty}_{-\infty}\int\limits^{\infty}_{-\infty}W(N(x+yi))\widetilde{e}\left(- t(x+yi)\right)\dif x \dif y, \quad t \geq 0 \quad \mbox{and} \quad
   g(k,n) =\sum_{r \bmod n}\chi(r)\widetilde{e}\left(\frac {kr}{n}\right).
\end{align*}
\end{lemma}
\begin{proof}
   We first recall the following Poisson summation formula for
   $\mz[i]$ (see the proof of \cite[Lemma 4.1]{G&Zhao}), which is itself an easy consequence of the classical Poisson summation formula in
$2$ dimensions:
\begin{align*}
   \sum_{j \in \mz[i]}f(j)=\sum_{k \in
   \mz[i]} \ \int\limits^{\infty}_{-\infty}\int\limits^{\infty}_{-\infty}f(x+yi)\widetilde{e}\left( -k(x+yi) \right)\dif x \dif y.
\end{align*}
   We get
\begin{align*}
    \sum_{m \in \mz[i]}\chi(m) & W\left(\frac {aN(m)}{X}\right) = \sum_{r \bmod n}\chi(r)\sum_{j \in \mz[i]}W\left(\frac {aN(r+jn)}{X}\right)  \\
   =& \sum_{r \bmod n}\chi(r) \sum_{k \in
   \mz[i]} \ \int\limits^{\infty}_{-\infty}\int\limits^{\infty}_{-\infty}W\left(\frac {aN(r+(x+yi)n)}{X}\right)\widetilde{e}\left(-k(x+yi) \right) \dif
   x \dif y.
\end{align*}
   We make a change of variables in the integral, writing
\begin{equation*}
   \sqrt{N\Big(\frac {n}{k}\Big )}\frac {k}{n}\frac {(r+(x+yi)n)}{\sqrt{X/a}}=u+vi,
\end{equation*}
   with $u,v \in \mr$. (If $k=0$, we omit the factors involving $k/n$.) With the Jacobian of
this transformation being $aN(n)/X$, we find that
\begin{equation*}
\begin{split}
  \int\limits^{\infty}_{-\infty}\int\limits^{\infty}_{-\infty} & W\left(\frac {aN(r+(x+yi)n)}{X}\right)\widetilde{e}\left(-k(x+yi)\right)\dif x \dif y \\
 &  = \frac {X}{aN(n)}\widetilde{e}\left(\frac {kr}{n}\right)\int\limits^{\infty}_{-\infty}\int\limits^{\infty}_{-\infty}
   W( N(u+vi))\widetilde{e}\left (-(u+vi)\sqrt{N \left( \frac{k}{n} \right) \frac{X}{a}} \right) \dif u \dif v,
   \end{split}
\end{equation*}
   whence
\begin{align*}
    \sum_{m \in \mz[i]}W\left(\frac {aN(m)}{X}\right)\chi(m) &= \frac {X}{aN(n)}\sum_{k \in
   \mz[i]}\widetilde{W}\left(\sqrt{\frac {N(k)X}{aN(n)}}\right)\sum_{r \bmod n}\chi(r)\widetilde{e}\left(\frac {kr}{n}\right).
\end{align*}
   As the inner sum of the last expression above is $g(k,n)$ by definition, this completes the proof of the lemma.
\end{proof}

    From Lemma \ref{Poissonsum}, we readily deduce the following
\begin{corollary}
\label{Poissonsumformodd} Let $n \in \mz[i],   n \equiv 1 \pmod {(1+i)^3}$ and $\chi$ a quadratic or quartic character  $\pmod {n}$. For any Schwartz class function $W$,  we have
\begin{align*}
   \sum_{\substack {m \in \mz[i] \\ (m,1+i)=1}}\chi(m)W\left(\frac {N(m)}{X}\right)=\frac {X}{2N(n)}\chi(1+i)\sum_{k \in
   \mz[i]}(-1)^{N(k)} g(k,n)\widetilde{W}\left(\sqrt{\frac {N(k)X}{2N(n)}}\right).
\end{align*}
\end{corollary}
\begin{proof}
     It follows from Lemma \ref{Poissonsum} that
\begin{equation}
\label{2.21}
\begin{split}
  \sum_{\substack {m \in \mz[i] \\ (m,1+i)=1}} & \chi(m)W \left( \frac {N(m)}{X} \right) =\sum_{m}\chi(m)W \left( \frac {N(m)}{X} \right) -\chi(1+i)\sum_{m} \chi(m)W \left( \frac {2N(m)}{X} \right) \\
  &=\frac {X}{N(n)}\sum_{k \in
   \mz[i]}g(k,n)\widetilde{W}\left(\sqrt{\frac {N(k)X}{N(n)}}\right) -\chi(1+i)\frac {X}{2N(n)}\sum_{k \in
   \mz[i]}g(k,n)\widetilde{W}\left(\sqrt{\frac {N(k)X}{2N(n)}}\right).
   \end{split}
\end{equation}

    Using the relation (see \eqref{eq:gmult} and \eqref{2.7})
\begin{align*}
    g((1+i)k,n)=\overline{\chi}(1+i) g(k,n),
\end{align*}
    we can rewrite the first sum in the last expression of \eqref{2.21} as
\begin{align*}
   \sum_{k \in
   \mz[i]}g(k,n)\widetilde{W}\left(\sqrt{\frac {N((1+i)k)X}{2N(n)}}\right)  &= \chi(1+i)\sum_{k \in
   \mz[i]}g((1+i)k,n)\widetilde{W}\left(\sqrt{\frac {N((1+i)k)X}{2N(n)}}\right)  \\
   &=\chi(1+i)\sum_{\substack {k \in \mz[i] \\ 1+i | k}}g(k,n)\widetilde{W}\left(\sqrt{\frac {N(k)X}{2N(n)}}\right) .
\end{align*}
Substituting this back to last expression in \eqref{2.21}, we get the desired result.
\end{proof}

    Suppose that $W(t)$ is a non-negative smooth function supported on $(1,2)$, satisfying $W(t)=1$ for $t \in (1+1/U, 2-1/U)$ and $W^{(j)}(t) \ll_j U^j$ for all integers $j \geq 0$.
    One shows using integration by parts and our assumptions on $W$ that
\begin{align}
\label{bounds}
     \widetilde{W}^{(\mu)}(t) \ll_{j} U^{j-1}|t|^{-j}
\end{align}
    for all integers $\mu \geq 0$, $j \geq 1$ and all real $t$. \newline

On the other hand, evaluating $\widetilde{W}(t)$ with polar coordinate gives
\begin{equation*}
\begin{split}
    \widetilde{W}(t) =\int\limits_{\mr^2}\cos (2\pi t y)W(x^2+y^2) \dif x \dif y & =\int\limits^{\infty}_{0}\int\limits^{2\pi}_0\cos (2\pi t r \sin \theta )W(r^2)r \dif r \dif \theta \\
    & =\int\limits^{\sqrt{2}}_{1}\int\limits^{2\pi}_0\cos (2\pi t r \sin \theta )r \dif r \dif \theta+O \left( \frac 1{U} \right).
    \end{split}
\end{equation*}

    In particular, we have
\begin{align}
\label{w0}
    \widetilde{W}(0) =\pi+O(\frac 1{U}).
\end{align}

    Similarly, for any $j \geq 0$, we have
\begin{align}
\label{bounds1'}
    \widetilde{W}^{(j)}(t) \ll 1.
\end{align}

    As $\Phi(t)$ satisfies the assumptions on $W(t)$, the estimations \eqref{bounds} -- \eqref{bounds1'} are also valid for $\widetilde{\Phi}(t)$.  So in the sequel, we shall use these estimations for $\widetilde{\Phi}(t)$ without further justification.

\section{Proof of Theorem \ref{quadraticmainthm}}
\label{Section 3}
\subsection{Evaluation of  $C(X)$}

     We have
\begin{equation} \label{squarefreenum}
    \sumstar_{\substack {N(c) \leq X \\ (c, 1+i)=1}} 1 =\sum_{\substack {N(c) \leq X \\ (c, 1+i)=1}}\mu^2_{[i]}(c)=\sum_{\substack {N(c) \leq X \\ (c, 1+i)=1}}\sum_{\substack { d^2 | c \\ d \equiv 1 \bmod {(1+i)^3 }}}\mu_{[i]}(d)
    =\sum_{\substack { N(d) \leq \sqrt{X} \\ d \equiv 1 \bmod {(1+i)^3 }}}\mu_{[i]}(d) \sum_{\substack {N(c) \leq X/N(d^2) \\ (c, 1+i)=1}} 1,
\end{equation}
    where the ``$*$'' on the sum over $c$ means that the sum is restricted to square-free elements $c$ of $\mathbb{Z}[i]$. \newline

The best known result \cite{Huxley1} for the Gauss circle problem gives that
\begin{align*}
  \sum_{N(a) \leq x} 1 = \pi x+O(x^{\theta})
\end{align*}
with $\theta = 131/146$.  This implies that
\begin{align*}
     \sum_{\substack {N(c) \leq X/N(d^2) \\ (c, 1+i)=1}} 1 = \frac {\pi X}{2N(d^2)}+O \left( \left (\frac {X}{N(d^2)} \right)^{\theta} \right).
\end{align*}

Inserting the above into \eqref{squarefreenum}, we get
\begin{equation*}
    \sumstar_{\substack {N(c) \leq X \\ (c, 1+i)=1}} 1 =\pi X \sum_{\substack { N(d) \leq \sqrt{X} \\ d \equiv 1 \bmod {(1+i)^3 }}}\frac {\mu_{[i]}(d)}{2N(d^2)} +O\left( X^{\theta} \right) = \frac {2 \pi X}{3\zeta_{\mq(i)}(2)}+O\left( X^{\theta} \right).
\end{equation*}

    We conclude from this that as $X \rightarrow \infty$,
\begin{align*}
   \sumstar_{(c, 1+i)=1}\Phi \left( \frac {N(c)}{X} \right) \sim \# C (X) \sim \frac {2 \pi X}{3\zeta_{\mq(i)}(2)}.
\end{align*}

    It follows from this and Lemma \ref{lem2.1} that the left-hand side of \eqref{quaddensity} equals
\begin{align}
\label{3.01}
    \int\limits^{\infty}_{-\infty}\phi(t) \dif t-\frac 1{2}
   \int\limits^{\infty}_{-\infty}\hat{\phi}(u) \dif u-2\lim_{X \rightarrow \infty}\frac {S(X, Y;\hat{\phi}, \Phi)}{\# C(X)\log X },
\end{align}
    where
\begin{align*}
    S(X,Y; \hat{\phi}, \Phi) =
    \sumstar_{(c, 1+i)=1} \ \sum_{\substack{ N(\varpi) \leq Y \\ \varpi \equiv 1 \bmod {(1+i)^3}}} \frac {\chi_{i(1+i)^5c}(\varpi)\log N(\varpi)}{\sqrt{N(\varpi)}}\hat{\phi} \left( \frac {\log N(\varpi)}{\log X} \right) \Phi \left( \frac {N(c)}{X} \right).
\end{align*}
    Here $\hat{\phi}(u)$ is smooth and has its support contained in the interval $(-2+\varepsilon, 2-\varepsilon)$ for some $0<\varepsilon<1$. To emphasize this condition, we shall set $Y=X^{2-\varepsilon}$ and write the condition $N(\varpi) \leq Y$ explicitly throughout this section . \newline

    On the other hand, observe that when ${\hat \phi}$ is supported in $(-2,2)$, we
     have
\begin{align}
\label{3.02}
  \int\limits^{\infty}_{-\infty}\phi(t) \dif t-\frac 1{2}
   \int\limits^{\infty}_{-\infty}\hat{\phi}(u) \dif u+\frac {1}{2} \int\limits^{\infty}_{-\infty}(1-\chi_{[-1,1]}(u)) \hat{\phi}(u) \dif u=\int\limits^{\infty}_{-\infty}\phi(t)\left( 1-\frac {\sin (2 \pi t)}{2 \pi
   t} \right) \dif t.
\end{align}

   Comparing \eqref{3.01} with \eqref{3.02}, we see that in order to establish Theorem \ref{quadraticmainthm}, it suffices to show that for any Schwartz function $\phi$ with $\hat{\phi}$ supported in
$(-2+\varepsilon, 2-\varepsilon)$ for any $0<\varepsilon<1$,
\begin{equation}
\label{01.50}
  \lim_{X \rightarrow \infty} \frac{S(X, Y;\hat{\phi}, \Phi)}{X \log X}=-\frac {\pi }{6\zeta_{\mq(i)}(2)} \int\limits^{\infty}_{-\infty} \left( 1-\chi_{[-1,1]}(t) \right) \hat{\phi}(t) \dif t.
\end{equation}
\subsection{Expressions $S_M(X,Y; \hat{\phi}, \Phi)$ and $S_R(X,Y; \hat{\phi}, \Phi)$ }
    Let $Z >1$ be a real parameter to be chosen later and write
     $\mu_{[i]}^2(c)=M_Z(c)+R_Z(c)$ where
\begin{equation*}
    M_Z(c)=\sum_{\substack {l^2|c \\ N(l) \leq Z}}\mu_{[i]}(l) \; \quad \mbox{and} \; \quad  R_Z(c)=\sum_{\substack {l^2|c \\ N(l) >
    Z}}\mu_{[i]}(l).
\end{equation*}

   Define
\[ S_M(X,Y; \hat{\phi}, \Phi) =\sum_{(c, 1+i)=1}M_Z(c) \sum_{\substack{ \varpi \equiv 1 \bmod {(1+i)^3} \\ N(\varpi) \leq Y}} \frac {\log N(\varpi)}{\sqrt{N(\varpi)}}\leg {i(1+i)c}{\varpi}  \hat{\phi} \left( \frac {\log N(
   \varpi)}{\log X} \right) \Phi\left( \frac {N(c)}{X} \right),\]
    and
\[ S_R(X,Y; \hat{\phi}, \Phi)
=\sum_{(c, 1+i)=1}R_Z(c) \sum_{\substack{ \varpi \equiv 1 \pmod {(1+i)^3} \\ N(\varpi) \leq Y}} \frac {\log N(\varpi)}{\sqrt{N(\varpi)}}\leg {i(1+i)c}{\varpi}\hat{\phi} \left( \frac {\log N(
   \varpi)}{\log X} \right) \Phi\left( \frac {N(c)}{X} \right), \]
so that $S(X,Y; \hat{\phi}, \Phi)=S_M(X,Y; \hat{\phi}, \Phi)+S_R(X,Y; \hat{\phi}, \Phi)$. \newline

    Using standard techniques (see \eqref{error1} below), we can show that, by choosing $Z$ appropriately, $S_R(X,Y; \hat{\phi}, \Phi)$ is small. We now give another expression for $S_M(X,Y; \hat{\phi},\Phi)$ using Poisson summation. We write it as
\begin{equation*}
\begin{split}
     S_{M}(X,Y; \hat{\phi}, \Phi)   = \sum_{\varpi \equiv 1 \bmod {(1+i)^3}} & \frac {\log N(\varpi)}{\sqrt{N(\varpi)}} \leg{i(1+i)}{\varpi} \hat{\phi} \left( \frac {\log N(   \varpi)}{\log X} \right) \\
   & \times \sum_{\substack {N(l) \leq Z \\ l \equiv 1 \bmod {(1+i)^3}}} \mu_{[i]}(l)\leg {l^2}{\varpi}  \sum_{\substack {c \in \mz[i] \\ (c, 1+i)=1} } \leg {c}{\varpi}\Phi \left( \frac {N(cl^2)}{X} \right).
   \end{split}
\end{equation*}

    Applying Corollary \ref{Poissonsumformodd}, we obtain that
\begin{align*}
    \sum_{\substack {c \in \mz[i] \\ (c, 1+i)=1} } \leg {c}{\varpi}\Phi \left( \frac {N(cl^2)}{X} \right)
  & = \frac {X}{2N(l^2\varpi )} \leg {1+i}{\varpi}\sum_{k \in
   \mz[i]}(-1)^{N(k)}g_2(k, \varpi)\widetilde{\Phi}\left(\sqrt{\frac {N(k)X}{2N(l^2\varpi)}}\right) \\
  & = \frac {X}{2N(l^2)\sqrt{N(\varpi )}} \leg {1+i}{\varpi}\sum_{k \in
   \mz[i]}(-1)^{N(k)} \leg{ik}{\varpi} \widetilde{\Phi}\left(\sqrt{\frac {N(k)X}{2N(l^2\varpi)}}\right),
\end{align*}
  where the last equality above follows from Lemma \ref{Gausssum} by noting that
\[   g_2(k, \varpi)=\leg {ik}{\varpi}N(\varpi)^{1/2}. \]

   We can now recast $S_M(X,Y; \hat{\phi}, \Phi)$ as
\begin{equation}
\label{3.6}
\begin{split}
   S_{M}&(X,Y; \hat{\phi}, \Phi)  \\
      =& \frac {X}{2}\sum_{\substack {N(l) \leq Z \\ l \equiv 1 \bmod {(1+i)^3}}} \frac {\mu_{[i]}(l)}{N(l^2)} \sum_{k \in
   \mz[i]}(-1)^{N(k)} \\
    & \hspace*{2cm} \times \sum_{\varpi \equiv 1 \bmod {(1+i)^3}} \frac {\log N(\varpi)}{N(\varpi)}\leg {kl^2}{\varpi}\hat{\phi}\left( \frac {\log N(
   \varpi)}{\log X} \right) \widetilde{\Phi}\left(\sqrt{\frac {N(k)X}{2N(l^2\varpi)}}\right).
   \end{split}
\end{equation}

The remainder of this section is devoted to the evaluations of $S_R(X,Y; \hat{\phi},\Phi)$ and $S_{M}(X,Y; {\hat \phi}, \Phi)$.
\subsection{Estimation of $S_R(X,Y; \hat{\phi}, \Phi)$}
\label{sec 3.1}
We first seek a bound for
\begin{equation*}
  E(Y; \chi_{i(1+i)^5cl^2}, \hat{\phi}) =\sum_{\substack {N(\varpi) \leq Y \\ \varpi \equiv 1 \bmod {(1+i)^3} }} \frac {\chi_{i(1+i)^5cl^2}(\varpi)\log N(\varpi)}{\sqrt{N(\varpi)}}\hat{\phi}\left( \frac {\log N(\varpi)}{\log X} \right),
\end{equation*}
   with $Y \leq X^{2-2\varepsilon}, (cl, 1+i)=1, X \leq N(cl^2) \leq 2X$. \newline

  Note that $\chi_{i(1+i)^5cl^2}$ is non-principal. It follows from Lemma \ref{lem2.7} and partial summation that
\begin{equation} \label{3.1}
  E(Y;\chi_{i(1+i)^5cl^2}, \hat{\phi})=\int\limits^{Y}_1 \frac {1}{\sqrt{u}}\hat{\phi}\left( \frac {\log u}{\log X} \right) \dif O \left( u^{1/2}\log^{3} (u) \log N(cl^2) \right) \ll \log^{5} (X).
\end{equation}

   We then deduce that
\begin{equation} \label{error1}
\begin{split}
    S_R(X,Y; \hat{\phi}, \Phi) &=\sum_{\substack {N(l) > Z \\ l \equiv 1 \bmod {(1+i)^3} }} \mu_{[i]}(l)   \sum_{\substack {c \in \mz[i] \\ (c, 1+i)=1} } E(Y;\chi_{i(1+i)^5cl^2}, \hat{\phi})\Phi \left( \frac {N(cl^2)}{X} \right) \\
     &\ll \sum_{N(l)>Z} \ \sum_{X/N(l)^2 \leq N(c) \leq
   2X/N(l)^2} \log^{5}(X) \ll \frac {X \log^{5} X}{Z}.
\end{split}
\end{equation}

\subsection{Estimation of $S_{M}(X,Y; {\hat \phi}, \Phi)$,  the second main term}

    It follows from \eqref{3.6} that the sum in $S_{M}(X,Y; {\hat \phi}, \Phi)$ corresponding to $k=0$ is zero.  We show in what follows that terms $k=\square$ ($k$ is a square), $k \neq 0$ in $S_M(X,Y; \hat{\phi}, \Phi)$ contribute a second main term. Before we proceed, we need the following result:
\begin{lemma}
\label{Poissonsumoverk} For $y>0$,
\begin{align*}
   \sum_{\substack {k \in \mz[i] \\ k \neq 0}}(-1)^{N(k)}\widetilde{\Phi}\left(\frac {N(k)}{y}\right)=-\widetilde{\Phi}\left(0\right)+O \left( \frac {U^{2}}{y^{1/2}} \right).
\end{align*}
\end{lemma}
\begin{proof}
   Note that
\begin{align*}
   \sum_{\substack {k \in \mz[i] \\ k \neq 0}}(-1)^{N(k)}\widetilde{\Phi}\left(\frac {N(k)}{y}\right)=\sum_{k \in \mz[i]}(-1)^{N(k)}\widetilde{\Phi}\left(\frac {N(k)}{y}\right)-\widetilde{\Phi}\left(0\right)
\end{align*}
and
\begin{align*}
    \sum_{k \in \mz[i]}(-1)^{N(k)}\widetilde{\Phi}\left(\frac {N(k)}{y}\right)=2 \sum_{k \in \mz[i]}\widetilde{\Phi}\left(\frac {2N(k)}{y}\right)-\sum_{k \in \mz[i]}\widetilde{\Phi}\left(\frac {N(k)}{y}\right).
\end{align*}
   By taking $n=1$ in Lemma \ref{Poissonsum}, we immediately obtain that
\begin{equation*}
    2 \sum_{k \in \mz[i]}\widetilde{\Phi}\left(\frac {2N(k)}{y}\right) = y\sum_{j \in
   \mz[i]}\breve{\Phi}\left(\sqrt{\frac {N(j)y}{2}}\right) \qquad \mbox{and} \qquad   \sum_{k \in \mz[i]}\widetilde{\Phi}\left(\frac {N(k)}{y}\right)  = y\sum_{j \in
   \mz[i]}\breve{\Phi}\left(\sqrt {N(j)y}\right),
\end{equation*}
   where
\begin{align*}
   \breve{\Phi}(t) =\int\limits^{\infty}_{-\infty}\int\limits^{\infty}_{-\infty}\widetilde{\Phi}(N(u+vi))\widetilde{e}\left(- t(u+vi)\right)\dif u \dif v, \quad t \geq 0.
\end{align*}

    It follows that
\begin{align}
\label{2.200}
    \sum_{k \in \mz[i]}(-1)^{N(k)}\widetilde{\Phi}\left(\frac {N(k)}{y}\right)=y\sum_{\substack{ j \in
   \mz[i] \\ (j, 1+i)=1}}\breve{\Phi}\left(\sqrt{\frac {N(j)y}{2}}\right).
\end{align}

    We have, when $t>0$, via integration by parts (and noting \eqref{bounds})
\begin{align*}
     \breve{\Phi}(t) & =\int\limits_{\mr^2}\cos (2\pi t y)\widetilde{\Phi}(x^2+y^2) \ \dif x \dif y=-\frac 1{\pi t} \int\limits_{\mr^2}\sin (2\pi t y)\widetilde{\Phi}'(x^2+y^2)y \ \dif x \dif y \\
    & =-\frac 1{2(\pi t)^2} \int\limits_{\mr^2}\cos (2\pi t y) \left( \widetilde{\Phi}'(x^2+y^2)+\widetilde{\Phi}''(x^2+y^2)2y^2 \right) \ \dif x \dif y \\
    & =\frac 1{4(\pi t)^3} \int\limits_{\mr^2}\sin (2\pi t y) \left( \widetilde{\Phi}''(x^2+y^2)6y+\widetilde{\Phi}'''(x^2+y^2)4y^3 \right) \ \dif x \dif y.
\end{align*}
   We can evaluate the last integral above with polar coordinates.  Using \eqref{bounds1'} for $\widetilde{\Phi}''$, $\widetilde{\Phi}'''$ if $r \leq 1$ and using \eqref{bounds} with $j=3$ for $\widetilde{\Phi}''$, $\widetilde{\Phi}'''$ if $r > 1$, we get
\begin{align*}
     \breve{\Phi}(t) \ll \frac {U^{2}}{t^3}.
\end{align*}
    The lemma follows immediately from this bound and \eqref{2.200}.
\end{proof}

    By a change of variables $k \mapsto k^2$ and noting that $k^2=k_1^2$ if and only if $k =\pm k_1$, we see that the terms $k=\square$, $k \neq 0$ in $S_M(X,Y; \hat{\phi}, \Phi)$ contribute
\begin{align*}
   S_{M, \square}  (X,Y; \hat{\phi}, \Phi) & =\frac {X}{4}\sum_{\substack{N(l) \leq Z \\ l \equiv 1 \bmod {(1+i)^3}}} \frac {\mu_{[i]}(l)}{N(l^2)} \sum_{\substack {(\varpi , l )=1 \\ \varpi \equiv 1 \bmod {(1+i)^3} }} \frac {\log N(\varpi)}{N(\varpi)}\hat{\phi} \left( \frac {\log N(
   \varpi)}{\log X} \right) \\
    & \hspace*{5cm} \times \sum_{\substack {k \in
   \mz[i] , k \neq 0 \\ (k, \varpi)=1}}(-1)^{N(k)}\widetilde{\Phi}\left(N(k)\sqrt{\frac {X}{2N(l^2\varpi)}}\right) \\
   &=S_{\square}-S'_{\square},
\end{align*}
   where
\begin{equation*}
\begin{split}
  S_{\square} = \frac {X}{4}\sum_{\substack {N(l) \leq Z \\ l \equiv 1 \bmod {(1+i)^3}}} \frac {\mu_{[i]}(l)}{N(l^2)} \sum_{\substack {(\varpi, l)=1 \\ \varpi \equiv 1 \bmod {(1+i)^3} }} \frac {\log N(\varpi)}{N(\varpi)} & \hat{\phi} \left( \frac {\log N(
   \varpi)}{\log X} \right) \\
   & \times \sum_{\substack {k \in
   \mz[i] \\ k \neq 0}}(-1)^{N(k)}\widetilde{\Phi}\left(N(k)\sqrt{\frac {X}{2N(l^2\varpi)}}\right)
   \end{split}
   \end{equation*}
and
\begin{equation*}
\begin{split}
   S'_{\square} =\frac {X}{4}\sum_{\substack{N(l) \leq Z \\ l \equiv 1 \bmod {(1+i)^3} }} \frac {\mu_{[i]}(l)}{N(l^2)} \sum_{\substack {(\varpi, l)=1 \\ \varpi \equiv 1 \bmod {(1+i)^3} }} \frac {\log N(\varpi)}{N(\varpi)} & \hat{\phi} \left( \frac {\log N(
   \varpi)}{\log X} \right) \\
    & \times \sum_{\substack {k \in \mz[i] \\ k \neq 0}}(-1)^{N(k)}\widetilde{\Phi}\left(N(k) \sqrt{\frac {X N(\varpi)}{2N(l^2)}}\right).
   \end{split}
   \end{equation*}

   To estimate $S'_{\square}$, \eqref{bounds} with $j=2$ gives
\begin{align*}
    \sum_{\substack {k \in
   \mz[i] \\ k \neq 0}}(-1)^{N(k)}\widetilde{\Phi}\left(N(k) \sqrt{\frac {X N(\varpi) }{2N(l^2)}}\right) \ll
   \sum_{\substack {k \in
   \mz[i] \\ N(k) \geq 1}} \frac {U N(l^2)}{N^2(k) X N(\varpi)} \ll \frac {U N(l^2)}{X N(\varpi)}.
\end{align*}

    From this we deduce that
\begin{align*}
   S'_{\square} \ll UZ.
\end{align*}

   Now we further rewrite $S_{\square}=S_{\square,1}+S_{\square,2}$ where
\begin{equation*}
\begin{split}
S_{\square,1} = \frac {X}{4}\sum_{\substack {N(l) \leq Z \\ l \equiv 1 \bmod {(1+i)^3}}} \frac {\mu_{[i]}(l)}{N(l^2)} \sum_{\substack {\varpi \equiv 1 \bmod {(1+i)^3} \\ (\varpi, l)=1 \\ N(\varpi) \geq X/N(l^2) }} \frac {\log N(\varpi)}{N(\varpi)} & \hat{\phi} \left( \frac {\log N(
   \varpi)}{\log X} \right) \\
   & \times \sum_{\substack {k \in
   \mz[i] \\ k \neq 0}}(-1)^{N(k)}\widetilde{\Phi}\left(N(k)\sqrt{\frac {X}{2N(l^2\varpi)}}\right),
   \end{split}
   \end{equation*}
   and
\begin{equation*}
\begin{split}
 S_{\square,2} =\frac {X}{4}\sum_{\substack {N(l) \leq Z \\ l \equiv 1 \bmod {(1+i)^3} }} \frac {\mu_{[i]}(l)}{N(l^2)} \sum_{\substack {\varpi \equiv 1 \bmod {(1+i)^3} \\ (\varpi, l)=1 \\ N(\varpi )< X/N(l^2) }} \frac {\log N(\varpi)}{N(\varpi)} & \hat{\phi} \left( \frac {\log N(
   \varpi)}{\log X} \right) \\
   & \times \sum_{\substack {k \in
   \mz[i] \\ k \neq 0 }}(-1)^{N(k)}\widetilde{\Phi}\left(N(k)\sqrt{\frac {X}{2N(l^2\varpi)}}\right).
   \end{split}
   \end{equation*}

    To estimate $S_{\square,2}$, we use \eqref{bounds} with $j=2$ to get that
\begin{align*}
    \sum_{\substack {k \in
   \mz[i] \\ k \neq 0}}(-1)^{N(k)}\widetilde{\Phi}\left(N(k)\sqrt{\frac {X}{2N(l^2\varpi)}}\right) \ll
   \sum_{\substack {k \in
   \mz[i] \\ k \neq 0}}U\frac {N(l^2\varpi)}{N^2(k)X} \ll U\frac {N(l^2\varpi)}{X} .
\end{align*}

    From this and \eqref{mer}, we obtain
\begin{align*}
   S_{\square,2}\ll X U \log \log X.
\end{align*}

    We now apply Lemma \ref{Poissonsumoverk} to see that $S_{\square,1}=S_{\square,M}+S_{\square,R}$, where
\[ S_{\square,M} = -\widetilde{\Phi}(0)\frac {X}{4}\sum_{\substack {N(l) \leq Z \\ l \equiv 1 \bmod {(1+i)^3} }} \frac {\mu_{[i]}(l)}{N(l^2)} \sum_{\substack {\varpi \equiv 1 \bmod {(1+i)^3} \\ (\varpi, l)=1 \\ N(\varpi) \geq X/N(l^2) }} \frac {\log N(\varpi)}{N(\varpi)}\hat{\phi} \left( \frac {\log N(
   \varpi)}{\log X} \right) ,  \]
   and
\[  S_{\square,R}   \ll  X^{1+1/4}U^{2}\sum_{\substack {N(l) \leq Z \\ l \equiv 1 \bmod {(1+i)^3} }} \frac {1}{N(l^2)^{1+1/4}} \sum_{\substack {\varpi \equiv 1 \bmod {(1+i)^3} \\ N(\varpi) \geq X/N(l^2) }} \frac {\log N(\varpi)}{N(\varpi)^{1+1/4}} . \]

    Using \eqref{mer} again, we arrive at
\begin{align*}
   S_{\square,R}  \ll X U^{2} \log \log X.
\end{align*}

   Let $\omega(l)$ denote the number of distinct primes in $\mz[i]$ dividing $l$. It is well-known that for $N(l) \geq 3$,
\begin{align*}
   \omega(l) \ll \frac {\log N(l)}{\log \log N(l)}.
\end{align*}

    It follows that
\begin{align*}
    S_{\square,M} &= -\widetilde{\Phi}(0)\frac {X}{4}\sum_{\substack {N(l) \leq Z \\ l \equiv 1 \bmod {(1+i)^3} }} \frac {\mu_{[i]}(l)}{N(l^2)} \sum_{\substack {\varpi \equiv 1 \bmod {(1+i)^3} \\ N(\varpi) \geq X/N(l^2) }} \frac {\log N(\varpi)}{N(\varpi)}\hat{\phi} \left( \frac {\log N(
   \varpi)}{\log X} \right)+O(X).
\end{align*}

Agan utilizing \eqref{mer} again, we get that
\begin{align*}
    S_{\square,M} &= -\widetilde{\Phi}(0)\frac {X}{4}  \sum_{\substack {N(l) \leq Z \\ l \equiv 1 \bmod {(1+i)^3} }}  \frac {\mu_{[i]}(l)}{N(l^2)} \int\limits^{\infty}_{X/N(l^2)} \hat{\phi} \left( \frac {\log u}{\log X} \right) \dif \log u+O(X \log \log X) \\
    &= -\widetilde{\Phi}(0)\frac {X\log X}{4} \sum_{\substack {N(l) \leq Z \\ l \equiv 1 \bmod {(1+i)^3} }} \frac {\mu_{[i]}(l)}{N(l^2)} \int\limits^{\infty}_{1- \log N(l^2)/ \log X} \hat{\phi}(t) \dif t+O(X \log \log X) \\
    &= -\widetilde{\Phi}(0)\frac {X\log X}{4} \sum_{\substack {N(l) \leq Z \\ l \equiv 1 \bmod {(1+i)^3} }}  \frac {\mu_{[i]}(l)}{N(l^2)} \int\limits^{\infty}_{1} \hat{\phi}(t) \dif t+O(X \log \log X) \\
    &= -\widetilde{\Phi}(0)\frac {X\log X}{3\zeta_{\mq(i)}(2)} \int\limits^{\infty}_{1} \hat{\phi}(t) \dif t+O\left( \frac {X \log X}{Z}+ X \log \log X \right) \\
    &= -\widetilde{\Phi}(0)\frac {X\log X}{6\zeta_{\mq(i)}(2)} \int\limits^{\infty}_{-\infty} \left( 1-\chi_{[-1,1]}(t) \right) \hat{\phi}(t) \dif t+O\left( \frac {X \log X}{Z}+ X \log \log X \right) \\
    &= -\frac {\pi X\log X}{6\zeta_{\mq(i)}(2)} \int\limits^{\infty}_{-\infty} \left( 1-\chi_{[-1,1]}(t) \right) \hat{\phi}(t) \dif t+O\left( \frac {X \log X}{Z}+\frac {X \log X}{U}+ X \log \log X \right),
\end{align*}
   where the last equality follows from \eqref{w0}. \newline

   Gathering the estimations for $S'_{\square}$, $S_{\square,2}$, $S_{\square,M}$ and $S_{\square,R}$, we obtain that
\begin{align*}
   S_{M, \square}(X,Y; \hat{\phi}, \Phi)=-\frac {\pi X\log X}{6\zeta_{\mq(i)}(2)} \int\limits^{\infty}_{-\infty} & \left( 1-\chi_{[-1,1]}(t) \right) \hat{\phi}(t) \ \dif t \\
   &+O \left( \frac {X \log X}{Z}+\frac {X \log X}{U}+ X U^{2} \log \log X+UZ \right).
\end{align*}

\subsection{Estimation of $S_{M}(X,Y; {\hat \phi}, \Phi)$, the remainder}
    Now, the sums in $S_{M}(X,Y; {\hat \phi}, \Phi)$
corresponding to the contribution of $k \neq 0$, $\square$ can be written as
$XR/2$, where $R$ is
\begin{align*}
 \sum_{\substack {N(l) \leq Z \\ l \equiv 1 \bmod {(1+i)^3} }} \frac {\mu_{[i]}(l)}{N(l^2)} \sum_{\substack {k \in
   \mz[i] \\ k \neq 0, \square}} & (-1)^{N(k)} \sum_{\varpi \equiv 1 \bmod {(1+i)^3}} \frac {\log N(\varpi)}{N(\varpi)}\leg {kl^2}{\varpi}\hat{\phi}\left( \frac {\log N(
   \varpi)}{\log X} \right) \widetilde{\Phi}\left(\sqrt{\frac {N(k)X}{2N(l^2\varpi)}}\right).
\end{align*}

   We define $\chi_{kl^2}$ to be $\leg {kl^2}{\cdot}$. Similar to our discussions in Section \ref{sect: Kronecker}, when $k$ is not a square, $\chi_{kl^2}$ can be regarded as a non-principle Hecke character modulo $(1+i)^5kl^2$ of trivial infinite type.  Analogous to \eqref{3.1}, one has the bound
\begin{align*}
  E(Y;\chi_{kl^2}, \hat{\phi}) \ll \log^{4} (X(N(kl^2)+2)).
\end{align*}

   We then deduce, by partial summation, that
\begin{align*}
\begin{split}
 \sum_{\varpi \equiv 1 \bmod {(1+i)^3}} & \frac {\log N(\varpi)}{N(\varpi)}\leg {kl^2}{\varpi}\hat{\phi} \left( \frac {\log N(
   \varpi)}{\log X} \right) \widetilde{\Phi}\left(\sqrt{\frac {N(k)X}{2N(l^2\varpi)}}\right) \\
   & =\int\limits^{Y}_1\frac {1}{\sqrt{V}}\widetilde{\Phi}\left(\sqrt{\frac {N(k)X}{2N(l^2)V}}\right) \dif E \left( V;\chi_{kl^2}, \hat{\phi} \right)
    \\
   & \ll \log^{4} \left( X(N(kl^2)+2) \right)
    \left( \frac {1}{\sqrt{Y}}
   \left |\widetilde{\Phi}\left(\sqrt{\frac {N(k)X}{2N(l^2)Y}}\right) \right |  +  \int\limits^{Y}_{1}\frac
   {1}{V^{3/2}}\left |\widetilde{\Phi}\left(\sqrt{\frac {N(k)X}{2N(l^2)V}}\right) \right |  \dif V \right. \\
   & \hspace*{2in} \left. + \int\limits^{Y}_{1}\sqrt{\frac {N(k)X}{N(l^2)}}\frac
   {1}{V^{2}} \left |\widetilde{\Phi}'\left(\sqrt{\frac {N(k)X}{2N(l^2)V}}\right) \right | \dif V \right).
\end{split}
\end{align*}

This gives raise to
\begin{align*}
  R  \ll \sum_{\substack {N(l) \leq Z \\ l \equiv 1 \bmod {(1+i)^3}}}\frac {1}{N(l^2)} \left( R_1+R_2+R_3 \right),
\end{align*}
    where
\[
  R_1  = \frac {1}{\sqrt{Y}} \sum_{\substack {k \in
   \mz[i] \\ k \neq 0}}\log^{4} (X(N(kl^2)+2))\left |\widetilde{\Phi}\left(\sqrt{\frac {N(k)X}{2N(l^2)Y}}\right) \right | , \]
\[   R_2  = \int\limits^{Y}_{1}\frac
   {1}{V^{3/2}}\sum_{\substack {k \in
   \mz[i] \\ k \neq 0}}\log^{4} (X(N(kl^2)+2))
   \left |\widetilde{\Phi}\left(\sqrt{\frac {N(k)X}{2N(l^2)V}}\right) \right |  \dif V, \]
and
\begin{equation*}
   R_3 = \int\limits^{Y}_{1}\sqrt{\frac {X}{N(l^2)}}\frac
   {1}{V^{2}} \sum_{\substack {k \in
   \mz[i] \\ k \neq 0}}\log^{4} (X(N(kl^2)+2)) \sqrt{N(k)} \left |\widetilde{\Phi}'\left(\sqrt{\frac {N(k)X}{2N(l^2)V}}\right) \right |\dif V.
\end{equation*}

    Similar to the estimations done in Section 3.3 of \cite{G&Zhao2}, we use \eqref{bounds1'} when $VN(l^2)/X \geq 1, N(k) \leq VN(l^2)/X$ and \eqref{bounds} when $VN(l^2)/X \leq 1$ or $VN(l^2)/X \geq 1, N(k) \geq YN(l^2)/X$ with $j=3$ or $4$.  These estimates give
\begin{equation*}
  R \ll \frac {\log^4 X Z \sqrt{Y}U^{3}}{X}.
\end{equation*}

Thus we conclude that the contribution of $k \neq 0, \square$ is
\begin{equation}
\label{error4}
 \ll  \log^4 X Z \sqrt{Y}U^{3}.
\end{equation}

\subsection{Conclusion }  We now combine the bounds \eqref{error1},
\eqref{error4} and take $Y=X^{2-2\varepsilon}, Z=\log^5 X$ (recall
that $U=\log \log X$) with any fixed $\varepsilon>0$ to obtain
\begin{align*}
  S(X, Y;\hat{\phi}, \Phi )
&=-\frac {\pi X\log X}{6\zeta_{\mq(i)}(2)} \int\limits^{\infty}_{-\infty} \left( 1-\chi_{[-1,1]}(t) \right) \hat{\phi}(t) \ \dif t  \\
&\hspace*{1cm}+O\left( \frac {X \log X}{Z}+\frac {X \log X}{U}+ X U^{2} \log \log X+U Z +\frac {X \log^{5} X}{Z} +\log^4 X Z \sqrt{Y}U^{3} \right) \\
& = -\frac {\pi X\log X}{6\zeta_{\mq(i)}(2)} \int\limits^{\infty}_{-\infty} \left( 1-\chi_{[-1,1]}(t) \right) \hat{\phi}(t) \ \dif t+o\left( X\log X \right),
\end{align*}
 which implies \eqref{01.50} and this completes the proof of Theorem \ref{quadraticmainthm}.

\section{Proof of Theorem \ref{quarticmainthm}}

   The proof of Theorem \ref{quarticmainthm} is similar to that of Theorem \ref{quadraticmainthm}. We define
    \[ S(X,Y; \hat{\phi}, \Phi), \; S_M(X,Y; \hat{\phi}, \Phi), \; S_R(X,Y; \hat{\phi}, \Phi) \]
     in this section similar to those defined in Section \ref{Section 3}, with the necessary modifications on characters,  as we replace $\chi_{i(1+i)^5c}$ by $\chi_{(1+i)^7c}$ in the definition of $S(X,Y; \hat{\phi}, \Phi)$ and we replace
$\leg {i(1+i)c}{\varpi}$ by $\leg {(1+i)^3c}{\varpi}_4$ in the definition of $S_M(X,Y; \hat{\phi}, \Phi)$, $S_R(X,Y; \hat{\phi}, \Phi)$ here. We assume that $\hat{\phi}(u)$ is smooth and has its support contained in the interval $(-20/19+\varepsilon, 20/19-\varepsilon)$ for some $0<\varepsilon<1$. We also set $Y=X^{20/19-\varepsilon}$ throughout this section. Similar to the proof of Theorem \ref{quadraticmainthm}, we see that in order to establish Theorem \ref{quarticmainthm}, it suffices to show that
\begin{align*}
  \lim_{X \rightarrow \infty} \frac{S(X, Y;\hat{\phi}, \Phi)}{X \log X}=0.
\end{align*}

  The relation $S(X,Y; \hat{\phi}, \Phi)=S_M(X,Y; \hat{\phi}, \Phi)+S_R(X,Y; \hat{\phi}, \Phi)$ and the estimation \eqref{error1} for $S_R(X,Y;
    \hat{\phi}, \Phi)$ is still valid. To estimate $S_M(X,Y; \hat{\phi}, \Phi)$, we recast it as
\begin{align*}
  S_{M}(X,Y; \hat{\phi}, \Phi)   = \sum_{\varpi \equiv 1 \bmod {(1+i)^3}} & \frac {\log N(\varpi)}{\sqrt{N(\varpi)}} \leg{(1+i)^3}{\varpi}_4\hat{\phi} \left( \frac {\log N(
   \varpi)}{\log X} \right) \\
    & \hspace*{1cm}  \times \sum_{\substack {N(l) \leq Z \\ l \equiv 1 \bmod {(1+i)^3} }} \mu_{[i]}(l)\leg {l^2}{\varpi}_4   \sum_{\substack {c \in \mz[i] \\ (c, 1+i)=1} } \leg {c}{\varpi}_4\Phi \left( \frac {N(cl^2)}{X} \right). \nonumber
\end{align*}

    Applying Corollary~\ref{Poissonsumformodd} , we obtain that
\begin{align*}
   \sum_{\substack {c \in \mz[i] \\ (c, 1+i)=1} } \leg {c}{\varpi}_4\Phi \left( \frac {N(cl^2)}{X} \right) & =\frac {X}{2N(l^2\varpi )} \leg {1+i}{\varpi}_4\sum_{k \in
   \mz[i]}(-1)^{N(k)}g_4(k, \varpi)\widetilde{W}\left(\sqrt{\frac {N(k)X}{2N(l^2\varpi)}}\right) \\
   & =\frac {X}{2N(l^2\varpi )} \leg {1+i}{\varpi}_4\sum_{k \in
   \mz[i]}(-1)^{N(k)} \overline{\leg{k}{\varpi}}_4 g_4(\varpi)\widetilde{W}\left(\sqrt{\frac {N(k)X}{2N(l^2\varpi)}}\right),
\end{align*}
    as one checks easily from Lemma \ref{quarticGausssum} that
\begin{align*}
    g_4(k, \varpi)=\overline{\leg {k}{\varpi}}_4 g_4(\varpi).
\end{align*}

   We can now rewrite $S_M(X,Y; \hat{\phi}, \Phi)$ as
\begin{equation}
\label{4.02}
\begin{split}
    & S_{M}(X,Y; \hat{\phi}, \Phi)  \\
   =& \frac {X}{2}\sum_{\substack {N(l) \leq Z \\ l \equiv 1 \bmod {(1+i)^3} }} \frac {\mu_{[i]}(l)}{N(l^2)} \sum_{\substack{ k \in
   \mz[i] \\ k \neq 0}}(-1)^{N(k)} \\
   & \hspace*{2cm} \times \sum_{\varpi \equiv 1 \bmod {(1+i)^3}} \frac {\log N(\varpi)}{N(\varpi)^{3/2}}\overline{\leg {kl^2}{\varpi}}_4 g_4(\varpi)\hat{\phi} \left( \frac {\log N(\varpi)}{\log X} \right)\widetilde{W}\left(\sqrt{\frac {N(k)X}{2N(l^2\varpi)}}\right).
   \end{split}
\end{equation}

\subsection{Average of quartic Gauss sums at prime arguments}
\label{section: smooth Gauss}

     For any ray class character $\chi \pmod {16}$, we let
\begin{align*}
   h(r,s;\chi)=\sum_{\substack{(n,r)=1 \\ n \equiv 1 \bmod {(1+i)^3} }}\frac {\chi(n)g_4(r,n)}{N(n)^s}.
\end{align*}

   The following lemma, a consequence of \cite[Lemma, p. 200]{P}, gives the analytic behavior of $h(r,s;\chi)$ for $\Re(s) >1$.
\begin{lemma}{\cite[Lemma 2.5]{G&Zhao1}}
\label{lem1} The function $h(r,s;\chi)$ has meromorphic continuation to the entire complex plane. It is holomorphic in the
region $\sigma=\Re(s) > 1$ except possibly for a pole at $s = 5/4$. For any $\varepsilon>0$, letting $\sigma_1 = 3/2+\varepsilon$, then for $\sigma_1 \geq \sigma \geq \sigma_1-1/2$, $|s-5/4|>1/8$, we have
\[ h(r,s;\chi) \ll N(r)^{\frac 12(\sigma_1-\sigma+\varepsilon)}(1+t^2)^{ \frac {3}{2}(\sigma_1-\sigma+\varepsilon)}, \]
  where $t=\Im(s)$. Moreover, the residue satisfies
\[ \mathrm{Res}_{s=5/4}h(r,s;\chi) \ll N(r)^{1/8+\varepsilon}. \]
\end{lemma}

   We derive from the above lemma the following:
\begin{lemma}
\label{lem3} Let $(b, 1+i)=1$. For any $d  \in \mz[i]$, we have
\begin{align} \label{lem4.3eq}
 \sum_{\substack {N(c) \leq x \\ c \equiv 1 \bmod {(1+i)^3} \\ c \equiv 0 \bmod {b}}} \overline{\leg {d}{c}}_4 g_4(c)N(c)^{-1/2}
 \ll N(d)^{1/10}N(b)^{-3/5}x^{4/5+\varepsilon}+N(d)^{1/8+\varepsilon}N(b)^{-1/2+\varepsilon}x^{3/4+ \varepsilon} .
\end{align}
\end{lemma}
\begin{proof}
   This follows essentially from the proof of \cite[Proposition 1, p. 198]{P} with a few modifications. Using the notations in \cite{P}, in the inclusion-exclusion type estimation of $\psi$ (the first expression below \cite[Lemma, p. 200]{P}), one needs to replace the estimation for a corresponding $h((b/\delta)^{n-2},s;\chi)$ by $h(d(b/\delta)^{n-2},s;\chi)$ (using \eqref{eq:gmult}).  Lemma \ref{lem1} then yields (with $n=4$ in our case)
\begin{align*}
  \psi & \ll N(d)^{(n-2)(3/2-\Re(s)+\varepsilon)/4}N(b)^{3n/4-1-n \Re(s)/2+2\varepsilon}(1+|s|^2)^{\text{Card}\sum_{\infty}(k) \cdot  (n/2-1/2)(3/2-\Re(s)+\varepsilon)}, \\
  \text{Res}_{s=1+1/n}\psi & \ll N(d)^{ (n/4-1/2)(1/2-1/n+\varepsilon)}N(b)^{n/4-3/2+\varepsilon}.
\end{align*}

  Using this and Perron's formula, we see that the ``horizontal" integrals can be estimated by observing that the estimates for the integrand (with $\sigma=\Re(s)$)
\begin{align*}
   N(d)^{(n/4-1/2)(1-\sigma+\varepsilon)}N(b)^{n/2-1-n \sigma/2+\varepsilon}T^{(n-1) \text{Card}\sum_{\infty}(k) (1-\sigma+\varepsilon)-1}X^{\sigma}
\end{align*}
   are monotonic in $\sigma$ and can therefore be estimated as the sum of the values at $\sigma=c- 1/2$ and $\sigma=c$, where $c=1+\varepsilon$. Thus, we obtain that the ``horizontal" integrals are
   \begin{align*}
\ll  N(d)^{(n-2)/8}N(b)^{n/4-1}T^{(n/2-1/2) \text{Card}\sum_{\infty}(k) -1}X^{1/2+\varepsilon} + N(b)^{-1}T^{-1}X^{1+\varepsilon}.
\end{align*}

The ``vertical" integral can be estimated by
\begin{align*}
   \ll N(d)^{ (n/8-1/4)}N(b)^{n/4-1} & X^{1/2+\varepsilon} \int\limits^{T}_{-T}(1+t^2)^{(n/4-1/4) \text{Card}\sum_{\infty}(k) -1/2} \dif t \\
   & \ll N(d)^{(n/8-1/4)}N(b)^{n/4-1}T^{(n/2-1/2) \text{Card}\sum_{\infty}(k)}X^{1/2+\varepsilon}.
\end{align*}

   Together, these estimates yield that the left-hand side of \eqref{lem4.3eq} is
\begin{align*}
 \ll N(d)^{(n/8-1/4)}N(b)^{n/4-1}& T^{(n/2-1/2) \text{Card}\sum_{\infty}(k)}X^{1/2+\varepsilon} \\
 & +N(b)^{-1}T^{-1}X^{1+\epsilon}+
N(d)^{(n-2)(1/2-1/n+\varepsilon)/4}N(b)^{n/4-3/2+\varepsilon}X^{1/2+1/n+ \varepsilon} .
\end{align*}

   With $R=2+(n-1) \text{Card}\sum_{\infty}(k)$, we now take
\begin{align*}
   T= \left( \frac {X}{N(b)^{n/2}N(d)^{(n-2)/4}} \right)^{1/R}
\end{align*}
   to get the desired result.
\end{proof}

We now use this Lemma~\ref{lem3} instead of Proposition 1 of \cite[p. 198]{P} in the sieve identity in Section 4 of \cite{P} (note that in our case \cite[Proposition 2, p. 206]{P} is still valid) to get that
\begin{equation} \label{4.2}
\begin{split}
  \sum_{\substack {n \equiv 1 \pmod {(1+i)^3} \\ N(n) \leq x}} & \overline{\leg {kl^2}{n}}_4 \frac {g_4(n) \Lambda(n)}{\sqrt{N(n)}} \\
   & \ll x^{4\varepsilon} \left( N(kl^2)^{(n-2)/(4R)} \left( \frac {x}{u_3} \right)^{n/(2R)} x^{1-1/R} \right. \\
    & \hspace*{2cm} \left. +N(kl^2)^{(n/4-1/2)(1/2-1/n+\varepsilon)}\left( \frac {x}{u_3} \right)^{n/4-1/2}x^{1/2+1/n}+x^{1-1/20}+xu_1^{-1/5} \right)  \\
   &\ll x^{4\varepsilon} \left( N(kl^2)^{1/10}x^{1-1/10}+N(kl^2)^{1/8+\varepsilon}x^{1-1/8}+x^{1-1/20} \right).
\end{split}
\end{equation}
where we take $u_3=x/u_1, u_1=x^{1/4}/8$ as in \cite{P} and we note that $n=4$, $R=5$ in our case. \newline

Since $g_4(n)=0$ if $n$ is not square-free, the proper prime powers contribute nothing to the left-hand side of \eqref{4.2}.  Thus we have
\begin{align*}
   E(x;k, l) :=\sum_{\substack {N(\varpi) \leq x \\ \varpi \equiv 1 \bmod {(1+i)^3} }} & \overline{\leg {kl^2}{\varpi}}_4 \frac {g_4(\varpi) \Lambda(\varpi)}{\sqrt{N(\varpi)}} \\
   & \ll x^{\varepsilon} \left( N(kl^2)^{1/10}x^{1-1/10}+N(kl^2)^{1/8+\varepsilon}x^{1-1/8}+x^{1-1/20} \right).
\end{align*}

\subsection{Estimation of $S_{M}(X,Y; {\hat \phi}, \Phi)$}

   It follows from partial summation that
\begin{align*}
&  \sum_{\varpi \equiv 1 \bmod {(1+i)^3}} \frac {\log N(\varpi)}{N(\varpi)^{3/2}}\overline{\leg {kl^2}{\varpi}}_4 g_4(\varpi)\hat{\phi} \left( \frac {\log N(
   \varpi)}{\log X} \right) \widetilde{W}\left(\sqrt{\frac {N(k)X}{2N(l^2\varpi)}}\right) \\
   &=\int\limits^Y_1\frac {1}{V}\hat{\phi}\left( \frac {\log V}{\log X} \right) \widetilde{W}\left(\sqrt{\frac {N(k)X}{2N(l)^2V}} \right ) \dif E(V; k, l) \\
   & \ll \frac {E(Y; k, l)}{Y}\hat{\phi}\left( \frac {\log Y}{\log X} \right) \widetilde{W}\left(\sqrt{\frac {N(k)X}{2N(l)^2Y}} \right )+\int\limits^Y_1\frac {E(V; k, l)}{V^2}\left| \hat{\phi}(\frac {\log V}{\log X}) \widetilde{W}\left(\sqrt{\frac {N(k)X}{2N(l)^2V}} \right ) \right| \dif V\\
   & \hspace*{3cm} +\frac 1{\log X}\int\limits^Y_1\frac {E(V; k, l)}{V^2}\left| \hat{\phi}' \left( \frac {\log V}{\log X} \right) \widetilde{W}\left(\sqrt{\frac {N(k)X}{2N(l)^2V}} \right ) \right| \dif V \\
    & \hspace*{3cm} +\sqrt{\frac {N(k)X}{N(l)^2}}\int\limits^Y_1\frac {E(V; k, l)}{V^{5/2}}\left| \hat{\phi} \left( \frac {\log V}{\log X} \right) \widetilde{W}'\left(\sqrt{\frac {N(k)X}{2N(l)^2V}} \right ) \right| \dif V.
\end{align*}

    It follows that
\begin{equation*}
\begin{split}
   \sum_{\substack{ k \in
   \mz[i] \\ k \neq 0}}(-1)^{N(k)}\sum_{\varpi \equiv 1 \bmod {(1+i)^3}} \frac {\log N(\varpi)}{N(\varpi)^{3/2}}\overline{\leg {kl^2}{\varpi}}_4 g(\varpi)& \hat{\phi} \left( \frac {\log N(\varpi)}{\log X} \right) \widetilde{W}\left(\sqrt{\frac {N(k)X}{2N(l^2\varpi)}}\right) \\
   & \ll  R_1+R_2+R_3+R_4,
   \end{split}
\end{equation*}
    where
\begin{equation*}
  R_1 =\sum_{\substack {k \in
   \mz[i] \\ k \neq 0}} \frac {M(Y)}{Y}\hat{\phi}\left( \frac {\log Y}{\log X} \right) \widetilde{W}\left(\sqrt{\frac {N(k)X}{2N(l)^2Y}} \right ),
   \end{equation*}
   \begin{equation*}
   R_2 =\int\limits^Y_1\sum_{\substack {k \in
   \mz[i] \\ k \neq 0}} \frac {M(V)}{V^2} \left| \hat{\phi}\left( \frac {\log V}{\log X} \right) \widetilde{W}\left(\sqrt{\frac {N(k)X}{2N(l)^2V}} \right ) \right|  \dif V,
   \end{equation*}
   \begin{equation*}
   R_3 =\frac 1{\log X}\int\limits^Y_1\sum_{\substack {k \in
   \mz[i] \\ k \neq 0}}\frac {M(V)}{V^2}\left| \hat{\phi}'\left( \frac {\log V}{\log X} \right) \widetilde{W}\left(\sqrt{\frac {N(k)X}{2N(l)^2V}} \right ) \right|  \dif V,
   \end{equation*}
   and
   \begin{equation*}
   R_4=\int\limits^Y_1\sum_{\substack {k \in
   \mz[i] \\ k \neq 0}} \sqrt{\frac {N(k)X}{N(l)^2}}\frac {M(V)}{V^{5/2}}\left| \hat{\phi} \left( \frac {\log V}{\log X} \right) \widetilde{W}'\left(\sqrt{\frac {N(k)X}{2N(l)^2V}} \right ) \right| \dif V,
\end{equation*}
with
\[ M(W) = N(kl^2)^{1/10} W^{1-1/10+\varepsilon} + N(kl^2) W^{1-1/8+\varepsilon} + W^{1-1/20+\varepsilon} . \]
We use \eqref{bounds1'} when $VN(l^2)/X \geq 1, N(k) \leq VN(l^2)/X$ and \eqref{bounds} when $VN(l^2)/X \leq 1$ or $VN(l^2)/X \geq 1, N(k) \geq YN(l^2)/X$ with $j=4$ to arrive at
\begin{align*}
   R_4 & \ll \int\limits^{X/N(l^2)}_1\sum_{\substack {k \in
   \mz[i] \\ N(k) \geq 1}}\frac {M(V)}{V^{5/2}} \sqrt{\frac {N(k)X}{N(l)^2}}\left (\frac {N(l)^2V}{N(k)X} \right )^{2}U^{3} \ \dif V \\
   & \hspace*{3cm} +\int\limits^Y_{X/N(l^2)} \sum_{\substack {k \in
   \mz[i] \\ 0< N(k) \leq VN(l^2)/X}} \sqrt{\frac {N(k)X}{N(l)^2}}\frac {M(V)}{V^{5/2}} \dif V \\
   &\hspace*{3cm} +\int\limits^Y_{X/N(l^2)}\sum_{\substack {k \in
   \mz[i] \\  N(k) \geq VN(l^2)/X}} \frac {M(V)}{V^{5/2}}\sqrt{\frac {N(k)X}{N(l)^2}}\left (\frac {N(l)^2V}{N(k)X} \right )^{2}U^{3} \ \dif V \\
   & \ll  \frac {N(l^2)^{6/5}Y^{1+\varepsilon}U^{3}}{X^{11/10}}+\frac {N(l^2)^{5/4+\varepsilon}Y^{1+\varepsilon}U^{3}}{X^{9/8+\varepsilon}}+\frac {N(l^2)Y^{19/20+\varepsilon}U^{3}}{X}.
\end{align*}

   The estimations for $R_1$,$R_2$ and $R_3$ are similar. We then conclude from \eqref{error1} and \eqref{4.02} that
\begin{equation*}
   S(X,Y; {\hat \phi}, \Phi)
    \ll X \left( \frac {Z^{6/5}Y^{1+\varepsilon}U^{3}}{X^{11/10}}+\frac {Z^{5/4+\varepsilon}Y^{1+\varepsilon}U^{3}}{X^{9/8+\varepsilon}}+\frac {Z Y^{19/20+\varepsilon}U^{3}}{X}+\frac {\log^{5} X}{Z} \right) = o\left( X \log X \right),
\end{equation*}
   when $U=\log \log X$, $Z=\log^5 X$. This completes the proof of Theorem \ref{quarticmainthm}. \newline

\noindent{\bf Acknowledgments.} P. G. is supported in part by NSFC grant 11371043  and L. Z. by the FRG grant PS43707.  Parts of this work were done when P. G. visited the University of New South Wales (UNSW) in June 2017. He wishes to thank UNSW for the invitation, financial support and warm hospitality during his pleasant stay.

\bibliography{biblio}
\bibliographystyle{amsxport}

\vspace*{.5cm}

\noindent\begin{tabular}{p{8cm}p{8cm}}
School of Mathematics and Systems Science & School of Mathematics and Statistics \\
Beihang University & University of New South Wales \\
Beijing 100191 China & Sydney NSW 2052 Australia \\
Email: {\tt penggao@buaa.edu.cn} & Email: {\tt l.zhao@unsw.edu.au} \\
\end{tabular}

\end{document}